\documentclass[a4paper, 11pt]{amsart} 

\usepackage{xcolor}
\usepackage{amsmath}
\usepackage{amssymb}			
\usepackage{amsfonts}
\usepackage{amsthm}
\usepackage{mathtools}
\usepackage{enumitem}
\usepackage{mathabx,epsfig}

\usepackage[colorlinks=true, urlcolor=blue, linkcolor=red]{hyperref}
\usepackage{tikz-cd}

\usepackage{tikz}
  \usetikzlibrary{matrix}
  \usetikzlibrary{fit}
  \usetikzlibrary{backgrounds}
  \usetikzlibrary{arrows}
  \usetikzlibrary{shapes}

\def\QQ{{\mathbb Q}}
\def\AA{{\mathbb A}}
\def\CC{{\mathbb C}}
\def\RR{{\mathbb R}}

\def\ZZ{{\mathbb Z}}

\def\fp{{\mathfrak p}}

\def\GL{{\mathrm{GL}}}
\def\SL{{\mathrm{SL}}}
\def\Gi{{\Gamma^\mathrm{inn}}}
\def\Go{{\Gamma^\mathrm{out}}}
\def\Fi{{F^\mathrm{inn}}}

\title[On the Image of Automorphic Galois Representations]{On the Image of Automorphic Galois Representations} 
\author{Alireza Shavali} 

\theoremstyle{plain}
\newtheorem{theo}{Theorem}[section]
\newtheorem{prop}[theo]{Proposition}
\newtheorem{cor}[theo]{Corollary}
\newtheorem{lemma}[theo]{Lemma}

\theoremstyle{remark}

\newtheorem{rem}[theo]{Remark}

\theoremstyle{definition}
\newtheorem{defi}[theo]{Definition}
\newtheorem{exmp}[theo]{Example}
\newtheorem{conj}[theo]{Conjecture}

\begin{document} 
\maketitle

\begin{abstract}
    In this paper, we study extra-twists for automorphic representations of $\GL_n$ and use them to give a precise description of the image of the Galois representations associated with regular algebraic cuspidal automorphic representations of $\GL_3$ over totally real fields. We also formulate a conjecture for the $\GL_n$-case and show how it follows from some standard conjectures in the Langlands program. Finally, assuming the existence of a motive associated with the representation, we study the relation of our constructions with the Mumford-Tate group.
\end{abstract}

{\small \tableofcontents}
\renewcommand{\baselinestretch}{1}\normalsize

\section{Introduction}
From the viewpoint of the Langlands program, one expects a connection between three different worlds: Motives, Galois Representations, and Automorphic Representations. One expects to be able to translate different features of the objects from one world to the others. For instance, the motive having many symmetries (having a big endomorphism ring) should force the Galois representation to have small image (as the Mumford-Tate conjecture explains) and these both should translate to the associated automorphic representation having special properties. This line of study can be traced back to the work of Serre on Galois representations associated with rational elliptic curves and his famous open image theorem \cite{serre1972proprietes}. This vague philosophy is our main guide throughout this paper. 

We will mostly focus on the connection of the automorphic side and the Galois side. The picture is well-understood in the $\GL_2/\QQ$ case. Serre's result can also be viewed as determining the image of the Galois representations associated with weight 2 modular forms with rational Fourier coefficients. Later, Ribet realized that for general modular forms of weight 2, the endomorphism of the associated abelian variety being big, translates into the form having some kinds of symmetries which he called inner-twists \cite{ribet2006galois} and was able to generalize Serre's result to this context \cite{ribet1980twists}. Momose then worked out the higher weight case \cite{momose1981adic} and Nekov{\'a}{\v{r}} generalized their work to Hilbert modular forms \cite{nekovavr2012level}. 
The main goal of this paper is to generalize their work to the $\mathrm{GL}_3$-case. The main observation is that in this case, one should also take into account the so-called "outer-twists" of an automorphic representation. This is due to the fact that not every automorphism of $\mathrm{SL}_3$ is inner. Inner- and outer-twists of a representation together form a group that we call the group of extra-twists.

Let us state our main result. 
Let $K$ be a totally real field and let $\pi$ be a regular algebraic cuspidal automorphic representation of $\GL_3(\AA_K)$ that is of general type, i.e. it neither satisfies $\pi \simeq \pi \otimes \chi$ for a non-trivial Hecke character $\chi$, nor $\pi \simeq \pi^\vee \otimes \eta$ for any Hecke character $\eta$.
Let $\QQ(\pi)$ be its Hecke field with Galois closure $E$, $\rho_{\pi,p}:G_K \rightarrow \GL_3(\QQ(\pi)\otimes_{\QQ} \overline{\QQ_p})$ the $p$-adic Galois representation attached to $\pi$ and $\Gamma \subseteq \mathrm{Aut}(E)$ the group of extra-twists which will be defined later. Here is our main result:
\begin{theo}\label{main}
      Let $F=E^\Gamma$  be the field fixed by all extra-twists of $\pi$. Then there exists a finite extension $L/K$ and a semi-simple algebraic group $H_p$ defined over $F_p:=F\otimes_{\QQ}\QQ_p$ which is a form of $\mathrm{SL}_3$ (constructed using the extra-twists), such that $\rho_{\pi,p}(G_L)$ is contained in $H_p(F_p)\cdot\QQ_p^{\times} \subseteq \GL_3(E\otimes_{\QQ} \overline{\QQ_p})$ and it is open in the $p$-adic topology. 
\end{theo}

In other words, for all $p$, the algebraic group $(\mathrm{Res}^{F_p}_{\QQ_p}H_p)\cdot\mathbb{G}_{m,\QQ_p}$ is the connected component of the $\QQ_p$-Zariski closure of the image and the image is open in there. 

Assuming the functoriality conjectures of Langlands, one can go through the arguments in the proof of the above and see what assumptions are needed on $\pi$ to prove such a result for $\GL_n$, i.e. when are extra-twists enough to give a precise description of the image. This should intuitively mean that $\pi$ is not coming from any smaller group via a Langlands transfer. We make this precise in Section \ref{Section-Automorphic} and define automorphic representations of general type and prove a big image theorem for the Galois representation associated to those, assuming Langlands functoriality. 

The conjectures of Clozel \cite{clozel1991representations}, predict the existence of a motive $M_{\pi}$ over $K$ (with coefficients in an extension of $\QQ(\pi)$) attached to $\pi$. The Mumford-Tate conjecture for this motive tells us that the groups $H_p$ should come from a global object $H$ defined over the field $F$ from Theorem \ref{main}. 
Assuming the existence of such a motive, we will use the action of extra-twists on the Hodge structure to construct a group $H_{\infty}$ over $F\otimes_{\QQ}\RR$ that should be the Archimedean part of the Mumford-Tate group. We will also use this action on the rational Hodge structure to construct a global group $H$.
This group will contain the  (special) Mumford-Tate group and in particular its dimension (which is equal to the dimension of all the groups $H_p$ from the Theorem \ref{main}) bounds the dimension of the Mumford-Tate group from above. 
\bigbreak
\noindent\textbf{Methods.} We will use the group of extra-twists of the automorphic representation $\pi$ to define a 1-cocycle which gives the form $H$ of $\mathrm{SL}_3$.
To prove the openness we need to compare the Lie algebras. We first twist away the determinant and only focus on the semi-simple part of the Lie algebra. We use automorphic base change and induction for degree 3 extensions and regularity of the Galois representation to show that the Galois representation is strongly irreducible and deduce that the Lie algebra is reductive. Then, following Ribet, we base change to the algebraic closure and use Goursat's lemma. 
To apply Goursat's lemma we need two things. First, we will use the classification of semi-simple Lie subalgebras of $\mathfrak{sl}_3$ and Langlands functoriality for $\mathrm{sym}^2:\GL_2 \rightarrow \GL_3$. In the $\GL_3$-case, being essentially $\mathrm{sym}^2$ and essentially self-dual are equivalent. This is important since it enables us to go back and forth between the Lie algebras and the groups. 
Second, we need a description of the field fixed by inner-twists and the field fixed by all extra-twists of $\pi$ in terms of the coefficients of the characteristic polynomial of the Galois representation at Frobenius elements, which is provided in Lemma \ref{cent}.
At the end we can twist back the determinant and give a complete description of the image of the Galois representation up to $p$-adic openness. 
\bigbreak
\noindent\textbf{The structure of this paper.}
In Section \ref{Section-Extra-Twists}, we define inner- and outer-twists of automorphic representations and Galois representations for $\GL_n$ and study their basic properties. We will also give the construction of the group $H$ from theorem \ref{main} in an abstract setting. 

In Section \ref{Section-Image of Galois}, we consider a general $n$-dimensional Galois representation with trivial determinant satisfying a set of natural properties and compute the $\QQ_p$-Lie algebra of its image, assuming that the $\overline{\QQ}_p$-Lie algebra is big. 

In Section \ref{Section-Automorphic}, we apply these results to Galois representations coming from certain automorphic representations and prove our main theorem. We also show that if one assumes enough conjectures from the Langlands program, then this could be generalized to a conjecture for the $\GL_n$-case. 

Finally, in Section \ref{Section-Mumford-Tate}, assuming the existence of a motive attached to $\pi$ we study the relation of the extra-twists with the Mumford-Tate group of this motive and propose some conjectures.
\bigbreak
\noindent\textbf{Acknowledgments.}
First, I want to thank my advisor Gebhard Böckle for suggesting that I work on this problem and his constant support during my PhD studies. I am furthermore grateful to Andrea Conti and Judith Ludwig for many helpful conversations. I would also like to thank Gaëtan Chenevier for suggesting Lemma \ref{Chenevier}, which simplified some of the arguments, and Chun Yin Hui for suggesting part \ref{RemarkValidMain} of Remark \ref{RemarkValid}, which resulted in removing a density one condition on an earlier version of this work. 
This research was supported by the Deutsche Forschungsgemeinschaft (DFG)
through the Collaborative Research Centre TRR 326 Geometry and Arithmetic of Uniformized Structures, project number 444845124.

\section{Extra-Twists for $\GL_n$}\label{Section-Extra-Twists}
In this section we define extra-twists for certain automorphic representation and Galois representations and then give a construction of a form of an algebraic group from a 1-cocycle.

Let $K$ be a number field and let $\pi$ be a cuspidal automorphic representation of $\GL_n(\AA_K)$ such that:
\begin{itemize}
    \item $\pi$ is not self-twist, i.e. there is no Hecke character $\chi\neq 1$ such that $\pi \simeq \pi \otimes \chi$.
    \item If $n>2$, $\pi$ is not essentially self-dual, i.e. there is no Hecke character $\eta$ such that $\pi \simeq \pi^{\vee} \otimes \eta$. 
\end{itemize}
These two conditions are analogues to the condition of a modular form being non-CM in the work of Ribet and Momose. It turns out that these conditions are enough in the $\GL_3$-case to have a big image theorem for the associated Galois representation as we will see later, but of course not in the $\GL_n$-case. 

Let $\QQ(\pi) \subset \CC$ be the Hecke (number) field of $\pi$ and let $E$ be a number field containing $\QQ(\pi)$. 
In what follows we will use strong multiplicity-one for cuspidal automorphic representations of $\GL_n$, many times without mentioning it. 

\subsection{Inner and outer twists}\label{DefTwist}

\begin{defi}\label{extra-def}
    An ($E$-)extra-twist of the automorphic representation $\pi$ is either of the following two:
    \begin{enumerate}
        \item (An inner-twist) A pair $(\sigma, \chi)$ where $\sigma \in \mathrm{Aut} (E)$ and $\chi:\AA_K^\times/K^\times \rightarrow \CC^\times$ is a Hecke character, such that ${}^\sigma \pi \cong \pi \otimes \chi$.
        \item (An outer-twist) A pair $(\tau, \eta)$ where $\tau \in \mathrm{Aut} (E)$ and $\eta:\AA_K^\times/K^\times \rightarrow \CC^\times$ is a Hecke character, such that ${}^\tau \pi \cong \pi^\vee \otimes \eta$. 
    \end{enumerate}
\end{defi}

\begin{rem} We make three remarks about this definition. 
    \begin{enumerate}
        \item[(a)] The role of $E$ might seem a bit auxiliary and one might think it should be enough to take  $E=\QQ(\pi)$. But, making this slightly more general definition will help us on the Galois side when dealing with issues regarding the field of definition of automorphic Galois representations. Also it will be more convenient in sections \ref{Section-Image of Galois} and \ref{Section-Automorphic} to assume that $E$ is Galois over $\QQ$. 
        Since we will always fix $E$ to begin with, we will usually drop it from the notation. 
        \item[(b)] Notice that the Galois action in the above definition is on the coefficients. In particular, do not confuse an outer-twist with an essential conjugate self-dual of an automorphic representation over a CM field (e.g. as in \cite{barnet2014potential}). 
        \item[(c)] For a general reductive group, there should be a class of extra-twists for every automorphism of a (fixed) based root datum. This would also make sense on the Galois side since the automorphism group of the dual root datum is clearly the same.  
    \end{enumerate}
\end{rem}

One can similarly define the notion of extra-twists for Galois representations. 
Let $E$ be a number field, $p$ a (rational) prime number and let $E_p = E \otimes_{\QQ} \QQ_p \cong \prod_{\fp|p} {E_\fp}$. Assume we have 
a continuous irreducible Galois representation:
$$
\prod_{\fp|p} \rho_{\fp} =
\rho : G_{K} \rightarrow \mathrm{GL}_n
({E_p})=\prod_{\fp|p} \mathrm{GL}_n
({E_\fp})
$$
unramified outside a finite set of places. We also assume that $\rho$ is neither self-twist nor essentially self dual in the $n>2$ case, i.e. it neither satisfies $\rho \simeq \rho \otimes \chi$ for a non-trivial Galois character $\chi$, nor $\rho \simeq \rho^\vee \otimes \eta$ for any Galois character $\eta$ in the $n>2$ case. 

\begin{defi}
An inner-twist of $\rho$ is a pair $(\sigma, \chi)$ where $\sigma \in \mathrm{Aut} (E)$ and $\chi:G_K \rightarrow {E_p}^\times$ is a (continuous) Galois character, such that ${}^\sigma \rho \cong \rho \otimes \chi$.
An outer-twist of $\rho$ is a pair $(\tau, \eta)$ where $\tau \in \mathrm{Aut} (E)$ and $\eta$ a Galois character, such that ${}^\tau \rho \cong \rho^\vee \otimes \eta$.
An extra-twist of $\rho$ is either an inner- or an outer-twist.
\end{defi} 
\begin{rem}
    Note that $\rho^\vee$ is just isomorphic to the representation $\rho^{-T}$ and hence has coefficients in $E_p$ (and not just $\overline{E}\otimes_{\QQ}\QQ_p$). This easily implies that the characters $\chi$ appearing in the extra-twists are forced to have values in $E_p^\times$ and we do not lose any generality by making this assumption in the definition.
\end{rem}

From now on, we assume that $K$ is totally real 
and $\pi$ is a regular algebraic cuspidal automorphic representation. Then it is known (\cite{harris2016rigid}, \cite{scholze2015torsion}) that there exists a compatible family of Galois representations $\rho_{\pi,p}$ associated with $\pi$. We will see in lemma \ref{Chenevier} that this compatible family can be defined over a coefficient field $E$ of finite degree and Galois over $\QQ$. Then we get a bijection between the set of $E$-extra-twists of $\pi$ and $E$-extra-twists of $\rho:=\rho_{\pi,p}$. Therefore, we usually identify the two.
\bigbreak
\noindent The most basic properties of the extra-twists of $\rho$ (or $\pi$) are summarized in the next lemma:

\begin{lemma}\label{BasicTwist} 
    Let $K$ be totally real and $\rho:G_K \rightarrow \GL_n(E_p)$ be a $p$-adic Galois representation that is neither self-twist nor essentially self-dual in the $n>2$ case. Then
    extra-twists of $\rho$ satisfy the following properties:
    \begin{enumerate}
        \item[(i)] For an extra-twist $(\sigma, \chi)$, the automorphism $\sigma$ uniquely determines the character $\chi$. 
        \item[(ii)] If $(\sigma, \chi)$ is an inner-twist and $(\tau, \eta)$ an outer-twist, then $\sigma \neq \tau$.
        \item[(iii)] Extra-twists form a group under the operation $(\sigma,\chi)\circ (\tau, \eta):=(\sigma\circ\tau, \chi \cdot {}^\sigma\eta)$.
        \item[(iv)] Inner-twists form a subgroup of the group of all extra-twists. If at least one outer-twist exists then this is an index 2 subgroup.
        \item[(v)] If $\rho$ is associated with an algebraic automorphic representation $\pi$ then
        for any inner-twist $(\sigma,\chi)$, the character $\chi$ is finite.
    \end{enumerate}
\end{lemma}
\begin{proof}
    First assume that $(\sigma, \chi_1)$ and $(\sigma, \chi_2)$ are both inner-twists. Then $\rho \simeq \rho \otimes \chi_1 \chi_2 ^{-1}$ which implies $\chi_1=\chi_2$ by our assumptions on $\rho$. A similar argument proves the other cases of \textit{(i)} and also \textit{(ii)}. For part \textit{(iii)} let assume that both 
    $(\sigma,\chi)$ and $(\tau, \eta)$ are inner-twists. Then 
    $$
    {}^\sigma({}^\tau \rho) \cong {}^\sigma (\rho\otimes \eta) = {}^\sigma\rho \otimes{}^\sigma\eta \cong \rho \otimes \chi \otimes {}^\sigma\eta
    $$
    The other cases are similar.
    For part \textit{(iv)} one just needs to notice that the product of two outer-twists is clearly an inner-twist. Finally, to see \textit{(v)}, note that since $K$ is totally real, the central character of $\pi$ must be of the form $|\cdot|^m\omega$ for some $m\in \ZZ$ and finite order character $\omega$.
    Hence $\det(\rho)=\epsilon_p^m \omega$ where $\epsilon_p$ is the $p$-adic cyclotomic character and we are viewing $\omega$ as a finite Galois character. Now, taking the determinant of both sides of
    $
    {}^\sigma\rho \cong \rho \otimes \chi
    $
    we get 
    $$\chi^n = \frac{{}^\sigma \det(\rho)}{\det(\rho)}=\frac{{}^\sigma \omega}{\omega} $$  
    which implies that $\chi^n$ (and hence $\chi$) is a finite order character. 
\end{proof}

We will denote the group of all extra-twists of a Galois representation (or an automorphic representation) by $\Gamma$ and the subgroup of inner-twists by $\Gi$ and the set of outer-twists by $\Go$, if there is any.
The last lemma shows that we can identify $\Gamma$ with a subgroup of $\mathrm{Aut}(E)$ by forgetting the character and we will do so from now on. 
Let $F:=E^{\Gamma}$ be the field fixed by all the extra-twists and $\Fi:=E^{\Gi}$ be the field fixed by the inner-twists. In particular, $\Gamma = \mathrm{Gal}(E/F)$ and $\Gi = \mathrm{Gal}(E/\Fi)$. If there is at least one outer-twist then $[\Fi:F]=2$, otherwise $F=\Fi$.

\subsection{Constructing forms of algebraic groups using 1-cocycles}\label{form}

In this section we explain the abstract construction of a form of an algebraic group from a $1$-cocycle of a Galois action of the automorphisms of the group. This will be used later for the group $\mathrm{SL}_n$ and a cocycle that comes from the group of extra-twists of a Galois representation. 

Let $E/F$ either be a finite Galois extension of fields or the semi-local Galois extension $E_p/F_p = (E/F) \otimes_{\QQ} \QQ_p$  for a finite Galois extension $E=F(\alpha)/F$ of number fields and let $\Phi$ be the minimal polynomial of $\alpha$. Let $G/E$ be an algebraic group and let $\Gamma = \mathrm{Gal}(E/F)$. Let $f:\Gamma \rightarrow \mathrm{Aut}_E(G)$ be a 1-cocycle and write $f_{\sigma}$ for the image of $\sigma$. We will construct a form of $G$ defined over $F$ using this cocycle.

Let $G' = \mathrm{Res}^E_{F}G$, so for every $F$-algebra $R$ we have 
$$G'(R)=G(E\otimes_{F}R)$$
Therefore, $G'(R)$ is equipped with an action of $\Gamma$
(where it acts on the first component of the tensor product)
which is 
clearly functorial. Hence the collection of morphisms
$\sigma:G'(R)\rightarrow G'(R)$ is a natural transformation
and so it is induced from a morphism $\sigma:G'\rightarrow G'$
of algebraic groups. Therefore, we can define $H$ to be the
subgroup of $G'$ satisfying $f_{\sigma}({}^\sigma g) =g$ for every $\sigma \in \Gamma$.
In other words $H=(G')^{tw_f(\Gamma)}$ where we define the $f$-twisted action of $\sigma$
on $G'$ to be given by ${}^{tw_f(\sigma)} g=f_{\sigma}({}^\sigma g)$.
This gives a closed subgroup of $G'$.

By base changing $H\subseteq \mathrm{Res}^E _{F}G$ to $E$ 
and then projecting to the id-component one gets
$$
H_E \hookrightarrow (\mathrm{Res}^E _{F}G)_E =
\prod_\Gamma G_E \xrightarrow{\pi_{\mathrm{id}}} G
$$
We prove that this is an isomorphsim by checking this on points. 
First, we need to give a description of the algebraic action of
$\Gamma$ on $\prod_\Gamma G$ via the identification 
$(\mathrm{Res}^E _{F}G)_E =
\prod_\Gamma G$. Let $R$ be an $E$-algebra. Note that
$\mathrm{Res}^E _{F}G(R) =
G(E\otimes_{F}R)$ and the algebraic action of $\Gamma$ is
just the action on the $E$ component. Now
$$
G(E\otimes_{F}R)=
G(E\otimes_{F} E \otimes_{E}R)
$$
and the action is only on the first $E$ component. Then
$$
G(E\otimes_{F}R)=
G(E\otimes_{F} E \otimes_{E}R)=
G\left(E\otimes_{F} \frac{F [x]}{\Phi(x)} \otimes_{E}R\right) =
G\left(\frac{E[x]}{\Phi(x)} \otimes_{E}R\right) 
$$
where the action is only on the coefficients of the first component.
So
$$
=G\left(\frac{E[x]}{\prod_{\Gamma}(x-\sigma\alpha)} \otimes_{E}R\right) 
=G\left(\left(\prod_{\Gamma}E\right) \otimes_{E}R\right) 
=G\left(\prod_{\Gamma}R\right)
=\prod_{\Gamma}G(R) 
$$
where the action of $\gamma \in \Gamma$ is given by 
$$
(a_\sigma)_\sigma \mapsto (\gamma a_{\gamma^{-1}\sigma})_\sigma
$$

\begin{prop}\label{GroupForm}
With the notation as above, for any $E$-algebra $R$ the map
$$
H(R) \hookrightarrow \mathrm{Res}^E _{F}G(R) =
\prod_\Gamma G(R) \rightarrow G(R)
$$
in an isomorphism of groups, so the algebraic 
group $H$ is a form of $G$.
\end{prop}
\begin{proof}
By definition, $H(R)$ is the subgroup of the elements invariant
under the twisted action of $\Gamma$. Let
$(g_\sigma)_\sigma \in \prod G(R)$ be invariant under the twisted
action:
$$
{}^{tw_f(\gamma)}(g_\sigma)_\sigma = 
(f_{\gamma}({}^\gamma g_{\gamma^{-1}\sigma}))_\sigma
$$ 
By looking at the component $\sigma = \gamma$ we get
$$
g_\gamma = f_{\gamma}({}^\gamma g_1)
$$
So the $g_1$ component determines all other $g_\gamma$'s.
This shows that the map 
$$
H(R) \hookrightarrow 
\prod_\Gamma G(R) \xrightarrow{\pi_{id}} G(R)
$$
is injective. 

To prove the surjectivety, we need to show that for every
$g_1 \in G(R)$ the element $(f_{\sigma}({}^\sigma g_1))_\sigma$
is invariant under the twisted action:
$$
{}^{tw_f(\gamma)}(f_{\sigma}({}^\sigma g_1))_\sigma
= (f_{\gamma}{}(^\gamma (f_{\gamma^{-1}\sigma}({}^{\gamma^{-1}\sigma} g_1))))_\sigma
= ((f_{\gamma} \circ {}^\gamma f_{\gamma^{-1}\sigma})({}^{\sigma} g_1))_\sigma
$$
Now, by the cocycle condition
$$
f_{\gamma} \circ {}^\gamma f_{\gamma^{-1}\sigma} = f_\sigma
$$
therefore
$$
{}^{tw(\gamma)}(f_{\sigma}({}^\sigma g_1))_\sigma =
(f_{\sigma}({}^\sigma g_1))_\sigma
$$
which is exactly what we needed to prove.
So the map is an isomorphism for all the $R$-points and hence
an isomorphism of affine algebraic groups over $E$.
\end{proof}
We can generalize this theorem to understand the behavior of $H$ under any Galois base change of $F$. 
\begin{cor}\label{formBaseChange}
    Let $F\subset F_0 \subset E$ be an intermediate field that is Galois over $F$ and let $\Gamma_0=\mathrm{Gal}(E/F_0)$ and $f_0:\Gamma_0 \rightarrow \mathrm{Aut}_E(G)$ be the restriction of $f$ to $\Gamma_0$. Then 
    $$
    H\times_{F}F_0 \simeq (\mathrm{Res}^E_{F_0} G)^{tw_{f_0}(\Gamma_0)}
    $$
\end{cor}
\begin{proof}
    Apply proposition \ref{GroupForm} to the group $(\mathrm{Res}^E_{F_0} G)^{tw_{f_0}(\Gamma_0)}$ for the  twisted action of the group $\Gamma/\Gamma_0$. 
\end{proof}

\section{Image of Galois Representations with Extra-Twists}\label{Section-Image of Galois}

The extra-twists of an algebraic automorphic representation induce extra-twists on the associated Galois representation. In this section we fix a $p$-adic Galois representation satisfying a list of natural properties (including the property that the $\overline{\QQ}_p$-Lie algebra of the image is big) and use these extra-twists to determine the $\QQ_p$-Lie algebra of the image of this Galois representation. In the next section, we will apply the results of this section to Galois representation associated to certain automorphic representations. We will need to use some automorphic input to show that in the $3$-dimensional case this list of properties is satisfied.

Let $K$ be a number field as usual and $G_K$ its absolute Galois group. 
Let $E$ be another number field that is assumed to be Galois over $\QQ$ and let $E_p = E \otimes_{\QQ} \QQ_p \cong \prod_{\fp | p} E_{\fp}$. Assume that for each finite place $\fp | p$ of $E$, we have a continuous semi-simple Galois representation 
$\rho_{\fp}:G_{K} \rightarrow \GL_n(E_\fp)$. It is usually more convenient to work with the product of all these Galois representations:
$$
\prod_{\fp | p} \rho_{\fp} =
\rho_{p} : G_{K} \rightarrow \mathrm{GL}_n
(E_p)=\prod_{\fp | p} \mathrm{GL}_n
(E_\fp)
$$
or equivalently a free $E_p$-module $V_p$ of rank $n$, with a continuous Galois action on it  and for each $\fp | p$
an $n$-dimensional vector space $V_\fp = V_p \otimes_{E_p} E_{\fp}$ over $E_\fp$
with a continuous Galois action such that
$V_p = \bigoplus_{\fp | p} V_{\fp}$ as $\QQ_{p}$-vector spaces. 

Each embedding $\lambda: E \hookrightarrow \overline{\QQ}_p$ induces an absolute value and hence gives a finite place $\fp$ of $E$ above $p$.
Therefore, $\lambda$ extends to an embedding $\lambda: E_{\fp} \hookrightarrow \overline{\QQ}_p$ by continuity. Now we define
$$
V_{\lambda}:= V_p \otimes_{E_p,\lambda \otimes id} \overline{\QQ}_p
= V_{\fp} \otimes_{E_{\fp},\lambda} \overline{\QQ}_p
$$
that is an $n$-dimensional vector space over 
$\overline{\QQ}_p$ with a continuous Galois action. We denote this representation by $\rho_{\lambda}$. Note that $\rho_\lambda:G_K\rightarrow \GL_n(\overline{\QQ}_p)$ is essentially the same object as $\rho_{\fp}:G_K \rightarrow \GL_n(E_\fp)$, it is just considered with coefficients in $\overline{\QQ}_p$ instead of $E_{\fp}$ via $\lambda$. 

Now, we need to make the following list of natural assumptions of our Galois representations to be able to compute the Lie algebra of the image in the next subsection. These properties are expected to hold for Galois representations attached to regular cuspidal algebraic automorphic representations of general type, after possibly some finite base change and twist by a character. 

\begin{defi}\label{valid}
    Keeping the above notations, the Galois representation $\rho_p:G_K \rightarrow \GL_n(E_p)$ is called \textbf{valid} if
    \begin{itemize}
    \item Each $\rho_\lambda$ is continuous and unramified outside a finite set $S$ of places of $K$ containing the Archimedean places and all places above $p$.
    \item $f^{(p)}_v(x):=\mathrm{CharPoly}(\rho_p(\mathrm{Frob}_v))$ has coefficients in $E$, for each place $v \notin S$. 
    %\item For each $\lambda$, $E$ is generated by $\mathrm{Tr}(\rho_\lambda(\mathrm{Frob}_v))$ (equivalentely, by $\mathrm{Tr}(\rho_p)$).
    \item Each $\rho_{\lambda}$ is neither self-twist nor essentially self-dual for $n>2$. 
    %\item Each $\rho_{\lambda}$ is strongly irreducible, i.e. the restriction of $\rho_\lambda$ to $G_L$, the absolute Galois group of $L$, is irreducible for any finite extension $L/K$.
    \item $\det(\rho_p)$ is trivial.
    \item For each $\lambda$ the $\overline{\QQ}_p$-Lie algebra of the $p$-adic Lie group $\rho_{\lambda}(G_K)$ is equal to $\mathfrak{sl}_n(\overline{\QQ}_p)$.  
\end{itemize}
\end{defi}

\begin{rem}\label{RemarkValid} We make the following four remarks about this definition: 
\begin{enumerate}
    \item Note that by the last condition,
    each $\rho_{\lambda}$ is strongly irreducible, i.e. the restriction of $\rho_\lambda$ to $G_L$, the absolute Galois group of $L$, is irreducible for any finite extension $L/K$. This is simply because going to a finite extension does not affect the Lie-algebra. In practice, we will usually need to prove this first in order to show that a Galois representation is valid. 
    \item Note that for $n=2$ we are not excluding essential self-duality but we are excluding being self-twist. 
    \item The condition on the determinant is not very restrictive because we can trivialize the determinant after a finite extension of $K$ and a twist.
    Since our first goal is to compute the semisimple part of the Lie algebra of the image, it doesn't change anything if we restrict to an open subgroup and also twist with a character. 
    \item \label{RemarkValidMain} In the case where $\rho_\lambda$'s come from a compatible family of semi-simple Galois representations, it is enough to check the last condition at only one $\lambda$. More precisely, by Theorem 3.19 and Remark 3.22 of \cite{hui2013monodromy}, the semi-simple rank and the formal character of the tautological representations of the algebraic monodromy group are independent of $\lambda$. Then \cite[Theorem 4]{larsen1990determining} implies that this uniquely determines the Lie algebra in the type $A_n$ case hence if we have $\mathfrak{sl}_n$ at one place $\lambda$, we should have $\mathfrak{sl}_n$ at every place. 
\end{enumerate}
\end{rem}

\subsection{Extra-twists and Galois representations}
From now on we assume that $\rho_p$ is a valid Galois representation. There are two cases that we have to deal with. $\rho_p$ either has an outer-twist or it does not. Notice that in the $n=2$ case we only have inner-twists since every representation is essentially self-dual. 
We try to deal with both cases at the same time but at some places it is easier to make a distinction. 
We need a few lemmas:

\begin{lemma}
    For every extra-twist $(\sigma, \chi_\sigma)$ of $\rho_p$, one has that $\chi_\sigma$ is a finite character.
\end{lemma}
\begin{proof}
    This is always true for inner-twists as we saw in lemma \ref{BasicTwist}. Let $(\sigma, \chi_\sigma)$ be an outer-twist. Then
    $$
    {}^\sigma \rho_p \cong \rho_p^\vee \otimes \chi_\sigma
    $$
    Looking at the determinants of both side we see that $\chi_\sigma ^n = 1$ which implies the result. 
\end{proof}

\begin{lemma}\label{charTwist}
    Let $L$ be a finite Galois extension of $K$ and let $V_1$ and $V_2$ be $p$-adic finite dimensional $G_{K}$-representations such that the restriction of $V_2$ to $G_L$ is absolutely irreducible. If $V_1 \simeq V_2$ as representations of $G_L$, then $V_1  \simeq V_2 \otimes \phi$ as representations of $G_K$ for some character $\phi$ of $G_K$.
\end{lemma}
\begin{proof}
    we can choose two bases for $V_1$
    and $V_2$ such that the representations
    $\rho_1:G_K \rightarrow \GL_n(\overline{\QQ_{p}})$
    and
    $\rho_{2}:G_K \rightarrow \GL_n(\overline{\QQ_{p}})$
    associated with $V_1$ and $V_2$
    are isomorphic when restricted to $G_L$. Now define
    $$
    \phi (g) := \rho_{1}^{-1}(g)\rho_{2}(g)
    $$
    A priori $\phi$ is just a map
    $\phi:G_K \rightarrow \GL_n(\overline{\QQ_{p}})$ which
    is trivial on $G_L$. We want to prove that it
    is actually a homomorphism with values
    in the center (hence actually a character) on all of $G_K$. 

    Let $g\in G_K$ and $h \in G_L$. 
    Note that
    $\rho_{1} (h) = \rho_{2} (h)$
    and 
     $\rho_{1} (ghg^{-1}) = \rho_{2} (ghg^{-1})$
     since $G_L$ is normal in $G_K$.
     Now the following computation shows that 
     $\phi (g)=\rho_{1}^{-1}(g)\rho_{2}(g)$
     commutes with $\rho_{2} (h)$:
     $$
     \rho_{1}^{-1}(g)\rho_{2}(g)\rho_{2} (h)
     = \rho_{1}^{-1}(g)\rho_{2}(gh) =
     \rho_{1}^{-1}(g)
     \rho_{2}(ghg^{-1})\rho_{2}(g)
     $$
     $$
     =\rho_{1}(g^{-1})
     \rho_{1}(ghg^{-1})\rho_{2}(g)
     =\rho_{1} (h)\rho_{1}^{-1}(g)\rho_{2}(g)
     = \rho_{2} (h)\rho_{1}^{-1}(g)\rho_{2}(g)
     $$
     Now since $V_2$ is absolutely irreducible when restricted to $G_L$, we have
     $\mathrm{End}_{G_L}(V_2)=\overline{\QQ_{p}}$ and
    we are done. 
\end{proof}

Now fix a valid Galois representation $\rho_p$. Note that by the last assumption in definition \ref{valid}, it makes sense to consider the group $\Gamma \subseteq \mathrm{Aut}(E)$ of all the extra-twists of $\rho_p$. Let $\Gi$, $F = E^\Gamma$, and $\Fi = E^\Gi$ be as in section \ref{DefTwist}. By lemma \ref{BasicTwist}, the character $\chi$ in an extra-twist $(\sigma,\chi)$ is uniquely determined by $\sigma$ so we use the notation $\chi_\sigma$ for this character.

We assumed in definition \ref{valid} that we know the $\overline{\QQ}_p$-Lie-algebra of the image. The next two lemmas are our main tool to compute the ${\QQ}_p$-Lie-algebra. The next lemma is the only place that we will use the assumption that $E/\QQ$ is Galois.

\begin{lemma}\label{cent}
    Let $G_L=\cap_{\sigma \in \Gamma} \ker (\chi_{\sigma})$ and let $L'/L$ be a finite extension that is Galois over $K$. For every finite unramified place $v$ of $L'$ let $a_v, b_v \in E$ such that $f_v(x):=\mathrm{charPoly}(\rho_p(\mathrm{Fr}_v)) = X^n-a_vX^{n-1}+ \cdots + (-1)^{n-1} b_vX+(-1)^n$. Then, $F^{inn}=\QQ(\{a_v\}_v)$ and $F=\QQ(\{a_v + b_v\}_v)$.
\end{lemma}

\begin{proof}
   We first prove that $f_v(x)\in F^{inn}[x]$. This is because for any inner-twist $(\sigma, \chi)$, the character $\chi$ becomes trivial after restricting to $G_{L'}$ so $\rho_p|_{G_{L'}} \cong {}^\sigma \rho_p|_{G_{L'}}$ which means $f_v$ is invariant under the action of $\Gamma_{inn}$ which implies the result. Now if $(\tau, \eta)$ is an outer-twist, after restriction to $G_{L'}$ one has $\rho_p|_{G_{L'}} \cong {}^\tau \rho^{-T}_p|_{G_{L'}}$. Looking at the characteristic polynomials of $\mathrm{Fr}_v$ of both sides one gets $X^n-a_vX^{n-1}+ \cdots + (-1)^{n-1} b_vX+(-1)^n = X^n-{}^\tau b_vX^{n-1}+ \cdots + (-1)^{n-1} {}^\tau a_vX+(-1)^n$. In particular, $a_v + b_v$ is invariant under the outer-twists as well. This implies $a_v+b_v \in F$. 
   
   Now, let $F'=\QQ(\{a_v\}_v) \subseteq F^{inn}$. We want to prove that $F'$ is the field fixed by the inner-twists. This suffices because $E$ was assumed to be Galois over $\QQ$ and hence it is Galois over $F'$.
   So it's enough to construct an inner-twist of $\rho_p$ for every $\sigma \in \mathrm{Gal}(\overline{\QQ}/F')$. Now, note that $\rho_p|_{G_{L'}}$ has traces in $F'$ so $\rho_p|_{G_{L'}}$ and ${}^\sigma\rho_p|_{G_{L'}}$ have the same traces and since they are semisimple, they must be conjugate. Now, by lemma \ref{charTwist} there exists a character such that $\rho_p \otimes \chi \simeq {}^\sigma \rho_p$ so we are done. 
   
   At last, let $F''=\QQ(\{a_v + b_v\}_v) \subseteq F^{inn}$. We want to prove that $F''$ is the field fixed by all extra-twists. So it's enough to construct an inner-twist or an outer-twist of $\rho_p$ for every $\sigma \in \mathrm{Gal}(\overline{\QQ}/F'')$. Now, note that $(\rho_p \oplus \rho^{-T}_p)|_{G_{L'}}$ has traces in $F''$ because the trace of the image of $\mathrm{Fr}_v$ would be $a_v+b_v$. So $(\rho_p \oplus \rho^{-T}_p)|_{G_{L'}}$ and ${}^\sigma(\rho_p \oplus \rho^{-T}_p)|_{G_{L'}}$ have the same traces and since they are semisimple, should be conjugate. 
   Now by the strong irreducibility assumption, we must either have 
    $\rho_p |_{G_{L'}} \simeq {}^\sigma\rho_p|_{G_{L'}}$
   or
   $\rho_p |_{G_{L'}} \simeq {}^\sigma\rho_p^{-T}|_{G_{L'}}$.
   Then by lemma \ref{charTwist} we either get an inner-twist or an outer-twist. 
\end{proof}

\begin{lemma}\label{isoRep}
    Let $\lambda, \mu: E \hookrightarrow \overline{\QQ}_p$ be two places of $E$ above $p$ and let $G_L=\cap_{\sigma \in \Gamma} \ker (\chi_{\sigma})$ and $L'$ a finite extension of $L$ that is Galois over $K$. Then in the case that there are outer-twists one has:
    \begin{enumerate}
        \item[i.] $V_{\lambda} \simeq V_{\mu}$ as representations of $G_{L'}$ if and only if $\lambda|_{F^{inn}}=\mu|_{F^{inn}}$
        \item[ii.] $V_{\lambda} \simeq V_{\mu}^\vee$ as representations of $G_{L'}$ if and only if 
    $\lambda|_{F}=\mu|_{F}$ but $\lambda|_{F^{inn}}\neq \mu|_{F^{inn}}$
    \end{enumerate}
    and in the case that there are no outer-twists one has that $F=F^{inn}=\QQ(\{a_v \}_v)$ and part (i) of the above is true and part (ii) never occurs. 
\end{lemma}
\begin{proof}
    Since all our representations are semi-simple, it is enough to check equality on characteristic polynomials and since they are continuous it is enough to check this on a dense subset. We check this on the Frobenius elements of finite places of $L'$ at which $\rho_p$ is unramified. 

    Keeping the notation of lemma \ref{cent}, the characteristic polynomial of $\mathrm{Fr}_v$ acting on $V_\lambda$ is $X^n-\lambda(a_v)X^{n-1}+\cdots+(-1)^{n-1}\lambda(b_v)X+(-1)^n$ and on $V_\mu$ is $X^n-\mu(a_v)X^{n-1}+\cdots+(-1)^{n-1}\mu(b_v)X+(-1)^n$. Since $a_v$'s generate $F^{inn}$ be lemma \ref{cent}, part i follows. For part ii, notice that
    the characteristic polynomial of $\mathrm{Fr}_v$ acting on $V_{\mu}^\vee$ is $X^n-\mu(b_v)X^{n-1}+\cdots+(-1)^{n-1}\mu(a_v)X+(-1)^n$. So if $V_{\lambda} \simeq V_{\mu}^\vee$ then $\lambda(a_v) = \mu(b_v)$ and $\lambda(b_v) = \mu(a_v)$ which implies that $\lambda(a_v+b_v) = \mu(a_v+b_v)$ and hence $\lambda|_{F}=\mu|_{F}$. On the other hand, $V_\lambda$ is not essentially self-dual so by lemma \ref{charTwist} its restriction to $G_{L'}$ cannot be self-dual. So $V_\lambda$ and $V_\mu$ are not allowed to be isomorphic as $G_{L'}$-representations in this case, hence $\lambda|_{F^{inn}}\neq \mu|_{F^{inn}}$. The other direction also follows easily since the characteristic polynomials of Frobenius elements at unramified places clearly match. 
\end{proof}

\subsection{The Lie algebra of the image}
Now, we want to compute the $\QQ_p$-Lie-algebra of the image of $\rho_p$. First, we need to use the results of section \ref{form} to construct the right algebraic group which contains the image and then compare the Lie-algebra of the image with the (algebraic) Lie-algebra of this group. 

Recall that we assumed that $\rho_p$ has trivial determinant. Therefore we have
$$
\rho_p : G_K \rightarrow \SL_n (E_p)
$$
We first define a 1-cocycle $f:\Gamma \rightarrow \mathrm{Aut}_{E_p}(\SL_n)$ using extra-twists. For every inner-twist $\sigma \in \Gamma^{inn}$ one has that $\rho_p|_{G_{L}}$ and ${}^\sigma \rho_p|_{G_{L}}$ have the same trace. Since each $\rho_\lambda$ is strongly irreducible, this means that they are isomorphic and there exists $\alpha_\sigma \in \SL_n(E_p)$ such that $\rho_p|_{G_{L}} = \alpha_\sigma \cdot {}^\sigma \rho_p|_{G_{L}} \cdot \alpha^{-1}_\sigma$. For the inner-twist $(\sigma,\chi_\sigma)$ we define $f_\sigma = \mathrm{ad}(\alpha_\sigma)$. 
If $\tau \in \Gamma^{out}$ is an outer-twist (if there exist any) then $\rho_p|_{G_{L}}$ and ${}^\tau \rho^{-T}_p|_{G_{L}}$ have the same traces and there exists $\alpha_\tau \in \SL_n(E_p)$ such that $\rho_p|_{G_{L}} = \alpha_\tau \cdot {}^\tau \rho^{-T}_p|_{G_{L}} \cdot \alpha^{-1}_\tau$.
For the outer-twist $(\tau,\chi_\tau)$ we define $f_\tau = \mathrm{ad}(\alpha_\tau) \circ (\cdot)^{-T}$. One can easily check that $f:\Gamma \rightarrow \mathrm{Aut}_{E_p}(\SL_n)$ defined above is a 1-cocycle. 

Now, as in section \ref{form} we can define the twisted action of $\Gamma$ by this cocycle on $\SL_n$. From the construction of $f$, it is clear that every matrix in the image of $G_L$ is invariant under this twisted action.Let $H = (\mathrm{Res}^{E_p}_{F_p} \SL_n)^{tw_f(\Gamma)}$ that is an algebraic group over $F_p$. Then it follows:
\begin{cor}\label{groupH}
     The representation $\rho_p|_{G_L}$ factors through $H(F_p)\subseteq \SL_n(E_p)$, i.e.  $\rho_p(G_L)\subseteq H(F_p)$. 
\end{cor}

Note that by proposition \ref{GroupForm}, $H$ is a form of $\SL_n$ and in particular is a semi-simple group. Also, note that since $G_L$ is open in $G_K$, the Lie-algebra of the $p$-adic Lie groups $\rho_p (G_K)$ and $\rho_p (G_L)$ are the same. Let $\mathfrak{g}$ be the Lie algebra (over $\QQ_p$) of $\rho_p (G_L)$ and $\mathfrak{h}$ be the Lie algebra of the algebraic group $H/F_p$, both viewed as Lie-subalgebras of $\SL_n(E_p)$. Our next goal is to show that these two Lie-algebras are in fact equal. 

\begin{prop}\label{LieAlgImage}
    With the notation as above, $\mathfrak{g}=\mathfrak{h}$.
\end{prop}
\begin{proof}
    First note that $\mathfrak{g} \subseteq \mathfrak{h}$ by the last corollary. Since $\mathfrak{h}$ is semi-simple, it suffices to prove that $\overline{\mathfrak{g}^{\mathrm{der}}} = \mathfrak{g}^{\mathrm{der}} \otimes_{\QQ_p} \overline{\QQ_p}$ is equal to $\overline{\mathfrak{h}} = \mathfrak{h} \otimes_{\QQ_p} \overline{\QQ_p}$.

    For every embedding $\sigma : F \hookrightarrow \overline{\QQ}_p$ fix an extension $\widetilde{\sigma} : E \hookrightarrow \overline{\QQ}_p$ of $\sigma$. All of the other extensions of $\sigma$ can be obtained by composing with different elements of the Galois group $\Gamma=\mathrm{Gal}(E/F)$, i.e. are of the form $\widetilde{\sigma}\tau$ for $\tau \in \Gamma$. Now we base change our representation $\rho_p|_{G_L}$
    to $\overline{\QQ_p}$ to get: 

\begin{tikzcd}
\rho_p:G_L \arrow[r] \arrow[rd] & \SL_n(E_p) \arrow[r, hook]        &  \SL_n(\overline{E_p}) \arrow[r, "="]       & \mathrm{Res}^E_F(\SL_n)(\overline{F_p})                              \\
                                & H(F_p) \arrow[r, hook] \arrow[u, hook] & H(\overline{F_p}) \arrow[r, "="] \arrow[u, hook] & \mathrm{Res}^E_F(\SL_n)^{tw(\Gamma)}(\overline{F_p}) \arrow[u, hook]
\end{tikzcd}
\bigbreak
 \noindent where $\overline{E_p}:=E_p \otimes_{\QQ_p} \overline{\QQ}_p$ and
  $\overline{F_p}:=F_p \otimes_{\QQ_p} \overline{\QQ}_p = \prod_{\sigma : F \hookrightarrow \overline{\QQ}_p}\overline{\QQ}_p$. Note that we have
    
    $$
    \mathrm{Res}^E_F(\SL_n)(\overline{F_p}) \simeq
    \prod_{\sigma:F \hookrightarrow \overline{\QQ_p}} \SL_n (E \otimes_{F,\sigma}\overline{\QQ_p}) \simeq
    \prod_{\sigma:F \hookrightarrow \overline{\QQ_p}} \SL_n (E \otimes_F E \otimes_{E,\widetilde{\sigma}}\overline{\QQ_p}) 
    $$
    $$
    \simeq 
    \prod_{\sigma:F \hookrightarrow \overline{\QQ_p}} \prod_{\Gamma} \mathrm{SL}_n(\overline{\QQ_p}) \simeq
    \prod_{\lambda:E \hookrightarrow \overline{\QQ_p}} \mathrm{SL}_n(\overline{\QQ_p})
    $$
\bigbreak
   \noindent where $\lambda = \widetilde{\sigma}\tau$ for $\tau \in \Gamma$. By proposition \ref{GroupForm}, projection to the id-component of $\Gamma$ gives the isomorphism of the form $\mathrm{Res}^E_F(\SL_n)^{tw(\Gamma)}$ of $\SL_n$ with $\SL_n$ over $E_p$ so we have: 
\bigbreak
\begin{tikzcd}
\rho_p:G_L \arrow[r] \arrow[rd] & \SL_n(E_p) \arrow[r, hook]        & \prod_{\sigma:F \hookrightarrow \overline{\QQ_p}} \prod_{\Gamma} \mathrm{SL}_n(\overline{\QQ_p}) \arrow[r, "="] & \prod_{\lambda:E \hookrightarrow \overline{\QQ_p}} \mathrm{SL}_n(\overline{\QQ_p}) \\
                                & H(F_p) \arrow[r, hook] \arrow[u, hook] & \prod_{\sigma:F \hookrightarrow \overline{\QQ_p}}  \mathrm{SL}_n(\overline{\QQ_p}) \arrow[u, hook]              &                                                                                   
\end{tikzcd}
\bigbreak
\noindent  For each embedding $\sigma : F \hookrightarrow \overline{\QQ}_p$
    the composition
    $$
    \rho_{\sigma}: G_{L} \rightarrow
    H(F_p) \hookrightarrow 
    \prod_{\sigma : F \hookrightarrow \overline{\QQ}_p} \mathrm{SL}_n(\overline{\QQ}_p)
    \xrightarrow{\mathrm{pr}_\sigma}
    \mathrm{SL}_n(\overline{\QQ}_p)
    $$
    corresponded to the action of $G_L$ on the vector space $V_{\lambda}$ for some embedding $\lambda:E \hookrightarrow \overline{\QQ}_p$ extending $\sigma$. Note that by lemma \ref{isoRep} these $V_\sigma$'s are neither isomorphic nor dual to each other after any finite extension. This is the main point of the rest of the argument. 

    On the level of Lie algebras this gives the embedding: $$\overline{\mathfrak{g}^{\mathrm{der}}} \subseteq \overline{\mathfrak{h}} \xrightarrow{\simeq}  \prod_{\sigma : F \hookrightarrow \overline{\QQ}_p} \mathfrak{sl}_n(\overline{\QQ}_p)$$
    Let
    $\mathfrak{g}^{{\mathrm{der}}}_{\sigma} \subseteq \mathfrak{sl}_n(\overline{\QQ}_p)$  be the projection of $\overline{\mathfrak{g}^{\mathrm{der}}}$ to the $\sigma$-component in of the above map. This is the $\overline{\QQ}_p$-Lie-algebra of the image of the representation $\rho_\sigma$ (=$\rho_\lambda$ for some $\lambda$ extending $\sigma$), so by our assumption on $\rho_p$ being valid, we have $\mathfrak{g}^{{\mathrm{der}}}_{\sigma} = \mathfrak{sl}_n(\overline{\QQ}_p)$. 

    Now we can apply \cite[Lemma 4.6]{ribet1980twists} to
    $$\overline{\mathfrak{g}^{\mathrm{der}}} \subseteq \overline{\mathfrak{h}} \rightarrow  \prod_{\sigma : F \hookrightarrow \overline{\QQ}_p} \mathfrak{sl}_n(\overline{\QQ}_p)$$
    We only need to prove that for every $\sigma, \tau : F \hookrightarrow \overline{\QQ}_p$ the projections 
    $(\mathrm{pr}_\sigma \times \mathrm{pr}_\tau) (\overline{\mathfrak{g}^{\mathrm{der}}})$ and $(\mathrm{pr}_\sigma \times \mathrm{pr}_\tau) (\overline{\mathfrak{h}})$ are equal. We follow the arguments of \cite[\S6.2]{serre1972proprietes}.
    
    Clearly it's enough to show that
    $(\mathrm{pr}_\sigma \times \mathrm{pr}_\tau) (\overline{\mathfrak{g}^{\mathrm{der}}}) = \mathfrak{sl}_n(\overline{\QQ}_p) \times \mathfrak{sl}_n(\overline{\QQ}_p)$. Note that the first factor corresponds to the representation $V_{\sigma}$
    and the second to $V_{\tau}$. Now by the Lie algebra version of the Goursat's theorem \cite[Lemma 5.2.1]{ribet1976galois} if $(\mathrm{pr}_\sigma \times \mathrm{pr}_\tau) (\overline{\mathfrak{g}^{\mathrm{der}}})$ is not equal to $\mathfrak{sl}_n(\overline{\QQ}_p) \times \mathfrak{sl}_n(\overline{\QQ}_p)$
    then it has to be the graph of an isomorphism. Let's call this automorphism
    $\phi:\mathfrak{sl}_n(\overline{\QQ}_p) \rightarrow \mathfrak{sl}_n(\overline{\QQ}_p)$. Since $\mathfrak{sl}_n$ is simple the group of its outer automorphisms is the group of automorphisms of its Dynkin diagram which is trivial for $n=2$ and is
    isomorphic to $\frac{\mathbb{Z}}{2\mathbb{Z}}$ in the $n>2$ case.
    In this case, the class of this non-trivial outer automorphism is given by the map
    $X\mapsto -X^{T}$. So $\phi$ is either an inner automorphism or a conjugate of this outer automorphism. 
    
    First assume $\phi$ is an inner automorphism and is given by conjugation by some  matrix $\alpha$. 
    In other words, we have the following diagram:

\begin{center} 
\begin{tikzcd}
                                     &  & \mathfrak{sl}_n(\overline{\QQ}_p) \arrow[dd, "\phi=ad(\alpha)"] \\
\overline{\mathfrak{g}} \arrow[rru, "pr_{\sigma}"] \arrow[rrd, "pr_{\tau}"] &  &                                                 \\
                                     &  & \mathfrak{sl}_n(\overline{\QQ}_p)                               
\end{tikzcd}
\end{center}  
which means that $V_{\sigma}$ and $V_{\tau}$ are isomorphic as representations of $\overline{\mathfrak{g}}$. This implies that they are indeed isomorphic as representations of some open subgroup of $G_L$ which contradicts Lemma \ref{isoRep}. 
Now assume that $\phi$ is a conjugate of $X\mapsto -X^{T}$. Similarly, this means that $V_{\sigma}\cong V_{\tau}^{\vee}$ as representations of some small enough normal open subgroups of $G_L$ which again contradicts lemma \ref{isoRep}. This implies the result. 
\end{proof}

\begin{cor}\label{bigImCor}
    The image of $G_L$ under $\rho_p$ is an open subgroup of the $p$-adic Lie group $H(F_p)$.
\end{cor}

If $\rho_p$ has no outer-twists, then the cocycle $f$ is always defined by an inner automorphism and $H$ is an inner-form of $\SL_n$. If there is at least one outer-twist then this is not true anymore, but $H$ becomes an inner-form after a degree two extension. In fact, the restriction of $f$ to the index two subgroup $\Gi$ factors through $\mathrm{Inn}_{E_p}(\SL_n) \subset \mathrm{Aut}_{E_p}(\SL_n)$. Then by corollary \ref{formBaseChange}, the base change of $H$ to $F^{\mathrm{inn}}_p = \Fi \otimes_{\QQ} \QQ_p$ is an inner-form. In any case we have:
\begin{cor}\label{split}
    The group $H_{F^{\mathrm{inn}}_p}$, the base change of $H$ to $F^{\mathrm{inn}}_p$, is an inner-form of the group $\SL_n$ which splits over $E_p$. Moreover, if $p$ splits in the extension $F^{\mathrm{inn}}_p/F_p$ ($p$ splits in $\Fi$ "more" than it does it $F$) then $H$ is an inner-form of $\SL_n$.
\end{cor}
\begin{proof}
    The first part follows from the discussion above. For the second part, if there are no outer-twists then there is nothing to prove. Otherwise,
    notice that when $p$ splits, $F^{\mathrm{inn}}_p \simeq F_p \times F_p$ so if $H$ becomes an inner-form over $F^{\mathrm{inn}}_p$, it was already an inner-form over $\Fi$. 
\end{proof}

\section{Image of Automorphic Galois Representations}\label{Section-Automorphic}
In this section we will apply the results of the last section to Galois representations attached to certain automorphic representations. In the  case of $\mathrm{GL}_2$ all the results are already known but we include this case as well, for the sake of completeness. Throughout this section, we assume that $K$ is a totally real number field and $\pi$ is a regular cuspidal algebraic automorphic representation of $\GL_n(\AA_K)$. It is known by the work of Harris-Lan-Taylor-Thorne \cite{harris2016rigid} or Scholze \cite{scholze2015torsion}, that there is a compatible family of $p$-adic Galois representations associated with $\pi$. Our goal is to understand the image of these representations. 

Let $|.|^m\omega$ be the central character of $\pi$ where $\omega$ is a finite order Hecke character and $m\geq 1$ an integer. Then for each embedding $\lambda:\QQ(\pi) \hookrightarrow \overline{\QQ_p}$ there exist a continuous semi-simple Galois representation
$$
\rho_{\pi,\lambda}:G_K \rightarrow \mathrm{GL}_n(\overline{\QQ_p})
$$
that is an unramified Galois representaton for all unramified places of $\pi$ not above $p$ and at all these places like $v$, the characteristic polynomial of $\mathrm{Frob}_v$ is determined by the Satake parameters of $\pi$ at $v$. In other words, these representations form a compatible family of Galois representations in the sense of Serre. Moreover, $\det (\rho_{\pi,\lambda}) = (\lambda\circ\omega) \cdot \epsilon_p^m$ where $\epsilon_p$ is the (Global $p$-adic) cyclotomic character and we are regarding the finite order character $\omega$ as a Galois character via class field theory. 

It is not known if we can conjugate these Galois representations to have values in the completions of the Hecke field $\QQ(\pi)$. 
Nevertheless, we will show that we can do this for a finite extension $E$ of $\QQ(\pi)$. This is exactly the reason why we defined extra-twists for any Galois coefficient field containing the Hecke field in definition \ref{extra-def}. 

The Galois representations $\rho_{\pi, \lambda}$ are expected to be irreducible since $\pi$ is cuspidal. Let $t_{\pi,\lambda}:G_K \rightarrow \overline{\QQ}_p$ be the trace of $\rho_{\pi,\lambda}$. This is an irreducible pseudo-representation and it clearly takes values in $\QQ(\pi)_{\lambda}$. 
Then by a result of Rouquier \cite[Theorem 5.1]{rouquier1996caracterisation} (or more generally Chenevier \cite[corollary 2.23]{chenevier2014p}) there exists a central simple algebra $D_\lambda$ over $\QQ(\pi)_\lambda$ of dimension $n^2$ such that this pseudo-representation can be realized as the reduced trace of a representation 
$G_K \rightarrow D_\lambda^\times$. The base change of this representation to $\overline{\QQ}_p$ clearly gives back $\rho_{\pi,\lambda}$ because of the Brauer-Nesbitt theorem. In other words, the image of $\rho_{\pi,\lambda}$ is in fact in $D_\lambda^\times$:
$$
\rho_{\pi, \lambda}:G_K \rightarrow D_\lambda^\times \subset D_\lambda^\times \otimes_{\QQ(\pi)_\lambda} \overline{\QQ}_p \simeq \GL_n(\overline{\QQ}_p) 
$$

\begin{lemma}[Chenevier] \label{Chenevier}
    Assume that $\rho_{\pi,\lambda}$ is irreducible for all $\lambda$ and regular for at least one $\lambda$. Then,
    there exists a finite  extension $E'/\QQ(\pi)$ that is Galois over $\QQ$, such that for all but finitely many finite places $\lambda$ of $\QQ(\pi)$ and any place $\mu$ of $E'$ above $\lambda$, the central simple algebra $D_\lambda$ splits over $E'_{\mu}$. In particular, there exists a finite extension $E/\QQ(\pi)$ such that all representations $\rho_{\pi,\lambda}$ can be defined over $E$ (can be conjugated to have values in completions of $E$).  
\end{lemma}
\begin{proof} First recall that a central simple algebra $D_\lambda \subset M_n(\overline{\QQ}_p)$ splits in an extension $M$ of $E_\lambda$ if and only if it contains an element with $n$ pairwise distinct eigenvalues in $M$. Let $v$ be a place of $K$ at which $\pi$ is unramified and $f^{(v)}(x) \in \QQ(\pi)[x]$ be the characteristic polynomial of the Frobenius element at $v$, which is independent of $\lambda$. Then as in the proof of \cite[Lemma 5.3.1]{barnet2014potential}, choosing a $\lambda$ for which $\rho_{\pi,\lambda}$ is regular, we get that for infinitely many places $v$ one has that $f^{(v)}(x)$ has distinct roots. This shows that if $E'$ is the splitting field of $f^{(v)}(x)$, then $D_\lambda$ splits over the completion of $E'$ at any finite place coprime to $v$ and the level of $\pi$. Since it clearly splits over some finite extension of $E_\lambda$ as well, we can find a number field $E$ which splits all $D_\lambda$'s at the same time. At the end we take the Galois closure over $\QQ$. 
    
\end{proof}
\begin{rem}\label{remRou}
    {A natural question that arises after this lemma is if one should expect $D_\lambda$'s to come from a global object $D/\QQ(\pi)$. We will discuss this more in the  section \ref{Section-Mumford-Tate}. In particular in the special case of $n\leq 3$ this follows from our results, existence of a motive associated to $\pi$ and the Mumford-Tate conjecture for that motive.}
\end{rem}

From now on we take $E$ to be the number field coming from the last lemma and we take our Galois representations to have values in $\GL_n(E_\mathfrak{p})$ for finite primes $\mathfrak{p}$ of $E$. So we are in the setting of section \ref{Section-Image of Galois} and we define:
$$
\prod_{\fp | p} \rho_{\pi,\fp} =
\rho_{\pi,p} : G_{K} \rightarrow \mathrm{GL}_n
(E_p)=\prod_{\fp | p} \mathrm{GL}_n
(E_\fp)
$$
where $E_p = E \otimes_{\QQ} \QQ_p \cong \prod_{\fp | p} E_{\fp}$ as usual. 

From now on, we assume that $\pi$ is neither self-twist nor essentially self-dual in the $n>2$ case. Then it makes sense to talk about $E$-extra-twists of $\pi$. Since $E$ is fixed we will drop it from the notation from now on. By multiplicity one, inner-twists of $\pi$ and $\rho_{\pi, p}$ agree (we are using class field theory to identify the characters). So let $\Gamma$, $\Gi$, $\Go$, $F$ and $\Fi$ be as in section \ref{Section-Extra-Twists}. 

The determinant of $\rho_p$ is given by $\omega \cdot \epsilon_p^m$. To apply the results of the last section, we first need to kill the determinant. This is always possible after a finite extension. In fact, after a finite extension the cyclotomic character will have values in $1+p\ZZ_p$ and then we can use the $p$-adic logarithm. Therefore, 
there exists a finite extension $M/K$ such that $\epsilon_p|_{G_{M}}$ has an $n$'th root. We fix one of these characters and denote it by $\epsilon_p^{1/n}$. Now, we enlarge $M$ to trivialize $\omega$ if necessary. Then 
$$\rho'_{\pi,p} := \rho_{\pi,p}|_{G_M} \otimes \epsilon_p^{-m/n}:G_M\rightarrow \GL_n(E_p)$$ has trivial determinant. This is the Galois representation that we will apply our results from the last section to. Notice that extra-twists of $\rho'_{\pi,p}$ and $\rho_{\pi,p}$ are the same. More precisely, the characters might have changed after the twist but the group $\Gamma \subseteq \mathrm{Aut}(E)$ has not. It is also not hard to see how the Lie algebra of the image changes:
\begin{lemma}\label{twistLieAlg}
    Assume that $\rho'_{\pi,p}$ is valid. Let $\mathfrak{g'}:=\mathrm{Lie}(\rho'_{\pi,p}(G_M))$ and $\mathfrak{g}=\mathrm{Lie}(\rho_{\pi,p}(G_L))$. Then $\mathfrak{g}^{\mathrm{der}}=\mathfrak{g}'$.
\end{lemma}
\begin{proof}
Let $G=\rho_{\pi,p}(G_M)$ and $G'=\rho'_{\pi,p}(G_M)$ and let $Z$ be the center of $\GL_n(E_p)$. Then we clearly have $G \subseteq G'\cdot Z$ and $G' \subseteq G\cdot Z$. Taking the Lie algebras we find that $\mathfrak{g} \subseteq \mathfrak{g'}+\mathfrak{z}$ and $\mathfrak{g'} \subseteq \mathfrak{g}+\mathfrak{z}$ where $\mathfrak{z}$ is the Lie algebra of the center. Since $\mathfrak{g'}$ is semi-simple by proposition \ref{LieAlgImage}, this implies that $\mathfrak{g}$ is reductive and $\mathfrak{g}^{\mathrm{der}}=\mathfrak{g'}$.
\end{proof}

In the case that $\rho'_{\pi,p}$ is valid, let $H_p/F_p$ be the semi-simple group from corollary \ref{groupH} applied to $\rho'_{\pi,p}$. Then we have:
\begin{prop}\label{validBigImage}
    If $\rho'_{\pi,p}$ is valid then there exists a finite extension $L$ of $K$ such that such that $\rho_{\pi,p}(G_L)$ is contained and $p$-adically open in $H_p(F_p)\cdot\QQ_p^{\times} \subseteq \GL_n(E_p)$. 
\end{prop}
\begin{proof}
    Since the image of $\rho'_{\pi,p}$ is contained in $H_p(F_p)$ after a finite extension, and the image of $\epsilon_p^{1/n}$ is in $\ZZ_p^{\times}$, the image of $\rho_{\pi,p}$ is contained in $H_p(F_p)\cdot\QQ_p^{\times}$ after a finite extension. The image is open in $H_p \subseteq \SL_n(E_p)$ by lemma \ref{twistLieAlg} and the image of the determinant is open in $\QQ_p^{\times}$, so we are done.  
\end{proof}
This in particular implies that the connected component of the $\QQ_p$-Zariski closure of the image is the algebraic group $(\mathrm{Res}^{F_p}_{\QQ_p}H_p)\cdot\mathbb{G}_{m,\QQ_p}$. Now, we only need to check validity of $\rho'_{\pi,p}$ to deduce that the image is big.

\subsection{The $\GL_2$-case}
As we mentioned in the introduction, essentially everything is known in this case by the work of Ribet \cite{ribet1980twists}, Momose \cite{momose1981adic} and Nekovar \cite{nekovavr2012level}. We repeat the arguments for the sake of completeness. In this case all representations are essentially self dual so there are no outer-twists and $\Gamma = \Gi$. One can in fact take $E=\QQ(\pi)$ (then $\Gamma$ would be abelian), but it is not necessary for our discussion. 
Recall the we assumed that $\pi$ is not self-twist, i.e. does not satisfy $\pi \simeq \pi \otimes \chi$ for $\chi \neq 1$. In this case it is more common to say that $\pi$ does not have complex multiplication (CM).  

\begin{prop}
    Assume that $n=2$ and $\pi$ does not have CM. Then $\rho'_{\pi,p}$ is valid.
\end{prop}
\begin{proof}
    Most of the properties are clear from the analogous properties for $\rho_{\pi,p}$. We only need to check strong irreducibility and compute the $\overline{\QQ}_p$-Lie algebra. 

    It is known that $\rho_{\pi,\lambda}$ is irreducible by a result of Ribet \cite{ribet2006galois}. It is also known that it is de Rham and the Hodge-Tate weights are distinct. 
    Strong irreducibility of $\rho'_{\pi,\lambda}$ and $\rho_{\pi,\lambda}$ are clearly equivalent. Assume that $\rho_{\pi,\lambda}|_{G_L}$ is reducible for some finite Galois extension $L/K$. Since $\rho_{\pi,\lambda}$ is semi-simple, so is its restriction to $G_L$ and we have $\rho_{\pi,\lambda}|_{G_L} \simeq \chi_1 \oplus \chi_2$. Since $\rho_{\pi,\lambda}$ is Hodge-Tate with distinct Hodge-Tate weights, so is $\rho_{\pi,\lambda}|_{G_L}$. This implies that $\chi_1 \neq \chi_2$. Let $K'$ be the fixed field of the stabilizer of $\chi_1$. Then it is clearly a degree 2 extension of $K$ and if $\mathrm{Gal}(K'/K)=\{1,\sigma \}$ then $\chi_2 = \sigma \chi_1$ (otherwise $\chi_1$ would be a direct summand of $\rho_{\pi,\lambda}$). This means that $\rho_{\pi,\lambda} \simeq \mathrm{Ind}^{K'}_K(\chi_1)$ which implies that $\pi$ is CM. Therefore $\rho_{\pi,\lambda}$ and hence $\rho'_{\pi,\lambda}$ are strongly irreducible. 

    Now let $\mathfrak{g}^{\mathrm{der}}_\lambda \subseteq \mathfrak{sl}_2(\overline{\QQ}_p)$ be the derived part of the $\overline{\QQ}_p$-Lie algebra of the image of $\rho'_{\pi,\lambda}:G_M \rightarrow \SL_2(\overline{\QQ}_p)$. Since $\rho'_{\pi,\lambda}$ is strongly irreducible, the irreducibility holds infinitesimally, i.e. $\mathfrak{g}_\lambda \subseteq \mathfrak{gl}_2(\overline{\QQ}_p)$ is an irreducible representation. This means that the centralizer of $\mathfrak{g}$ and hence its center is in the center of $\mathfrak{gl}_2(\overline{\QQ}_p)$ which implies that $\mathfrak{g}^{\mathrm{der}}_\lambda \subseteq \mathfrak{sl}_2(\overline{\QQ}_p)$ is also irreducible. The only irreducible semi-simple Lie-subalgebra of $\mathfrak{sl}_2$ is itself so we are done. 
    
\end{proof}
\noindent Now from proposition \ref{validBigImage} it follows:
\begin{cor}\label{bigGL2}
    Let $\pi$ be a regular, algebraic, cuspidal automorphic representation of $\GL_2(\AA_K)$ that doesn't have complex multiplication and let $F$ be the field fixed by inner-twists. Then there exists an inner form $H_p$ of $\SL_2$ over $F_p$ and a finite extension $L$ of $K$ such that the image of $\rho_{\pi,p}(G_L)$ is contained and $p$-adically open in $H_p(F_p)\cdot\QQ_p^{\times} \subseteq \GL_n(E_p)$.
\end{cor}

In the work of Ribet, Momose and Nekovar, they construct an Azumaya algebra $D_p/F_p$ which contains the image. The relation to the last corollary is that if $D_p^\times$ is the algebraic group of units of $D_p$ then $(D_p^{\times})^\mathrm{der}=H_p$. In the case that $\pi$ has parallel weight 2, so we expect an abelian variety to be associated to $\pi$, this algebra $D_p$ is closely related to the endomorphism ring of that abelian variety. We will explain the relation of these results to the Mumford-Tate conjecture for that abelian variety in section \ref{Section-Mumford-Tate}.

\subsection{The $\GL_3$-case}
In this section we prove our main result. 
The CM-case for $n=2$ can be thought of as the form essentially coming from $\mathrm{GL}_1$ by induction. Similarly, in the $n=3$ case we need to first exclude all the cases that $\pi$ comes from smaller groups via a Langlands transfer, in which case the image would be easy to describe by previous results.
It turns out that we only need to exclude the following two cases to be able to describe the image:
\begin{enumerate}
    \item $\pi$ is essentially $\mathrm{sym}^2$, i.e. there exists an automorphic representation $\theta$ of $\mathrm{GL}_2(\AA_K)$ and a Hecke character $\eta$ such that 
    $\pi = \mathrm{sym}^2(\theta) \otimes \eta$.
    \item $\pi$ is an induction of a character, i.e. there exist a degree 3 extension $L/K$ and a Hecke character $\eta$ of $\AA_L$
    such that $\pi=\mathrm{Ind}^L_K (\eta)$.
\end{enumerate}

Notice that Langlands functoriality is known for $\mathrm{sym}^2:\mathrm{GL}_2\rightarrow \mathrm{GL}_3$ by \cite{gelbart1978relation} and automorphic base change is known for prime degree extensions by \cite{arthur1989simple}.
In the first case above, determining the image reduces to the $\mathrm{GL}_2$ case and in the second case to the $\mathrm{GL}_1$ case. 
The next two lemmas give equivalent classifications of the above cases and show that they follow from our primary assumptions on $\pi$ (not essentially self-dual and not self-twist) that are needed to define extra-twists to begin with.
\begin{lemma}
    $\pi$ is essentially $\mathrm{sym}^2$ if and only if there exist a Hecke character $\chi$ such that $\pi = \pi^{\vee}\otimes \chi$.
\end{lemma}
\begin{proof}
    Since $\mathrm{GL}_2$ representation are essentially self-dual the "only if" part is clear. Now assume $\pi = \pi^{\vee}\otimes \chi$ and let $\omega$ be the central character of $\pi$. Taking the central characters of both sides we have $\chi^3 = \omega^2$. So $\chi = (\omega \chi^{-1})^2$ has a square root and by twisting out this square root we can assume that $\pi$ is self-dual. Now the result follows from \cite{ramakrishnan2014exercise}.
\end{proof}

\begin{rem}
    Since we assumed $K$ is totally real and $\pi$ is cuspidal regular algebraic and hence the existence of the associated Galois representations is known, one can equivalently work with the associated Galois representation since strong multiplicity one is known for $\mathrm{GL}_n$. Then one can give a different proof in the Galois side by investigating the projective image of the representation.  
\end{rem}
We also need the following lemma from \cite[\S 3]{arthur1989simple}
\begin{lemma}
    $\pi$ is an induction of a character if and only if there exists a Hecke character $\chi$ such that $\pi = \pi \otimes \chi$.
\end{lemma}

\begin{defi}\label{3general}
    An automorphic representation $\pi$ of $\GL_3(\AA_K)$ is said to be of \textbf{general type} if $\pi$ is neither  essentially self-dual nor self-twist.
    Equivalently, $\pi$ is neither essentially $\mathrm{sym}^2$ nor an induction of a character.
\end{defi}
Since we know strong multiplicity one for $\mathrm{GL}_n$ and Langlands functoriality for $\mathrm{sym}^2:\mathrm{GL}_2\rightarrow \mathrm{GL}_3$ and automorphic induction for degree three extensions, these assumptions are equivalent to the similar assumption on each of the Galois representations $\rho_{\pi,\lambda}$.

In \cite{bockle2024irreducibility}, Böckle and Hui prove that (in the $n=3$ case), $\rho_{\pi,\lambda}$ is irreducible for all $\lambda$. They also prove that for a density 1 set of rational primes $\mathcal{P}$, $\rho_{\pi,\lambda}$ is de Rham with distinct Hodge-Tate weights for all $\lambda:E \hookrightarrow \overline{\QQ_p}$ and $p \in \mathcal{P}$. We will use these results to check the validity of $\rho'_{\pi,p}$ (in fact we only need regularity at one place $\lambda$ by part 4 of Remark \ref{RemarkValid}).  
\begin{prop}
    Assume that $\pi$ is not self-twist. Then for each $p \in \mathcal{P}$ and $\lambda:E \hookrightarrow \overline{\QQ}_p$ one has that $\rho'_{\pi,\lambda}$ is strongly irreducible. 
\end{prop}
\begin{proof}
    Assume that $\rho_{\pi,\lambda}|_{G_L}$ is reducible for some finite Galois extension $L/K$. Since $\rho_{\pi,\lambda}$ is semi-simple (in fact irreducible), so is its restriction to $G_L$ and we have that $\rho_{\pi,\lambda}|_{G_L}$ decomposes into the sum of irreducible direct summands. If it decomposes into a two irreducible summand, then the action of $G_K$ cannot switch the two summands for dimension reasons and hence each summand is actually a sub-representation of $\rho_{\pi,\lambda}$ which is a contradiction. So we must have $\rho_{\pi,\lambda}|_{G_L} \simeq \chi_1 \oplus \chi_2 \oplus \chi_3$. Since $\rho_{\pi,\lambda}$ is Hodge-Tate with distinct Hodge-Tate weights, so is $\rho_{\pi,\lambda}|_{G_L}$. This implies that the three characters are distinct. Let $K'$ be the fixed field of the stabilizer of $\chi_1$. The action of $G_K$ on these characters must be transitive
    so $K'$ is a degree 3 extension of $K$ and if $\mathrm{Gal}(K'/K)=\{1,\sigma, \sigma^2 \}$ then $\chi_2 = \sigma \chi_1$ and $\chi_3 = \sigma^2 \chi_1$ (or the other way around). This means that $\rho_{\pi,\lambda} \simeq \mathrm{Ind}^{K'}_K(\chi_1)$ which implies that 
    $\pi$ is also a degree three automorphic induction of a character and hence is self-twist which contradicts the assumption.
    Therefore $\rho_{\pi,\lambda}$ and hence $\rho'_{\pi,\lambda}$ are strongly irreducible. 
\end{proof}

\begin{prop}
    Assume that $\pi$ is of general type. Then for each prime number $p$ and embedding $\lambda:E \hookrightarrow \overline{\QQ}_p$ one has that 
    the $\overline{\QQ}_p$-Lie algebra of the image of $\rho'_{\pi,\lambda}$ is $\mathfrak{sl}_3(\overline{\QQ}_p)$.
\end{prop}
\begin{proof}
    First assume that $p \in \mathcal{P}$. Let $V_\lambda$ be the underlying vector space of $\rho'_{\pi,\lambda}$ and let $\mathfrak{g}_\lambda \subseteq \mathfrak{sl}_3(\overline{\QQ}_p)$ be the $\overline{\QQ}_p$-Lie algebra of the image of $\rho'_{\pi,\lambda}:G_M \rightarrow \SL_3(\overline{\QQ}_p)$. Since $\rho'_{\pi,\lambda}$ is strongly irreducible, for every finite extension $N$ of $M$ we have $\mathrm{End}_{G_N}(V_\lambda)=\overline{\QQ}_p$, hence this should be true infinitesimally and we have $\mathrm{End}_{\mathfrak{g}_\lambda}(V_\lambda)=\overline{\QQ}_p$. This means that the standard representation $\mathfrak{g}_\lambda \hookrightarrow \mathrm{End}_{\overline{\QQ}_p}(V_\lambda)\simeq \mathfrak{gl}_3(\overline{\QQ}_p)$ is an irreducible faithful representation which implies that $\mathfrak{g}_\lambda$ is reductive. Let $\mathfrak{g}^{\mathrm{der}}_\lambda$ be its derived subgroup and hence a semi-simple subgroup of $\mathfrak{sl}_3(\overline{\QQ}_p)$. Since the center of $\mathfrak{gl}_3$ acts by scaler multiplication, any $\mathfrak{g}^{\mathrm{der}}_\lambda$-invariant subspace is automatically $\mathfrak{g}_\lambda$-invariant as well which implies that 
    the centralizer of $\mathfrak{g}$ and hence its center are in the center of $\mathfrak{gl}_3(\overline{\QQ}_p)$ which implies that $\mathfrak{g}^{\mathrm{der}}_\lambda \subseteq \mathfrak{gl}_3(\overline{\QQ}_p)$ is also irreducible. 

    The only irreducible, semi-simple Lie-subalgebras of $\mathfrak{sl}_3$ up to conjugation are $\mathfrak{sl}_3$ or $\mathfrak{sl}_2$ embedded into $\mathfrak{sl}_3$ by $\mathrm{sym^2}$. We need to show that the latter does not happen. Assume that $\mathfrak{g}^{\mathrm{der}}_\lambda$ is the image of $\mathrm{sym}^2:\mathfrak{sl}_2 \rightarrow \mathfrak{sl}_3$. Since $\mathrm{sym}^2$ is an irreducible representation, its centralizer in $\mathfrak{gl}_3$ is the center which means that  $\mathfrak{g}_\lambda = \mathfrak{g}^{\mathrm{der}}_\lambda \oplus \mathfrak{z}(\mathfrak{g}_\lambda)$ is in the image of $\mathrm{sym}^2:\mathfrak{gl}_2 \rightarrow \mathfrak{gl}_3$. This means that there is an open subgroup of $G_M$ whose image under $\rho'_{\pi,p}$ is in the image of $\mathrm{sym}^2:\GL_2 \rightarrow \GL_3$. So there exists a finite Galois extension $M'$ of $M$ such that $\rho'_{\pi,p}(G_{M'})$ is in the image of $\mathrm{sym}^2:\GL_2(E_\lambda) \rightarrow \GL_3(E_\lambda)$ and hence is essentially self-dual. Since the determinant of $\rho'_{\pi,p}$ is trivial, there exist a finite extension $N$ of $M'$ (which clearly can be taken to be Galois over $M$) such that the restriction to $G_N$ is in fact self-dual. 
    Now applying lemma  \ref{charTwist} to two representations $\rho'_{\pi,p}|_{G_N}$ and $\rho'_{\pi,p}|_{G_N}^\vee$, there exists a character $\phi$ such that $\rho'_{\pi,p}\simeq (\rho'_{\pi,p})^\vee \otimes \phi$ which contradicts non-essential-self-duality of $\pi$. This contradiction implies the result in the case where $p \in \mathcal{P}$.

    Now by \cite[Theorem 3.2]{bockle2024irreducibility}, the irreducible type of $\rho_{\pi,\lambda}$ is independent of $\lambda$. This in particular means that if the $\overline{\QQ}_p$-Lie algebra of the image of $\rho_{\pi,\lambda}$ contains $\mathfrak{sl_3}$ for one $\lambda$ (irreducible type $A_2$), it contains $\mathfrak{sl}_3$ for every $\lambda$. This clearly implies the result in general. 
\end{proof}

\noindent Now we can easily deduce our main result: 
\begin{proof}[Proof of Theorem \ref{main}]
    The last two propositions imply that for $\pi$ of general type, $\rho'_{\pi,p}$ is valid for any prime number $p$. Then proposition \ref{validBigImage} implies the result. 
\end{proof}

\begin{exmp}\label{VanTop}
    In \cite{van1994non}, van Geeman and Top construct a 3-dimensional $\QQ(i)$-rational compatible family of (motivic) Galois representations of $G_\QQ$ which is not self-twist or essentially self-dual and an automorphic representation of $\GL_3(\AA_\QQ)$ that should correspond to it. We can apply our results to the Galois representations they construct. For each prime $p$ they construct a Galois representation
    $$
    \rho_p:G_\QQ \rightarrow \GL_3(\QQ(i)\otimes \QQ_p)
    $$
    which has the property that $\rho_p \simeq \overline{\rho_p}^\vee \otimes \epsilon_p$ where $\overline{(\cdot)}$ indicates complex conjugation. For each unramified $p$, the characteristic polynomial of $\mathrm{Frob}_p$ looks like
    $$
    X^3-b_pX^2+p\overline{b_p}X-p^3
    $$
    and they give a list of values of $b_p \in \QQ(i)$ for small primes. 

    Now in our notation, the coefficient field is $E=\QQ(i)$. There is one outer twist $(\overline{\cdot}, \epsilon_p)$ and there can't be any more non-trivial extra-twists since $\mathrm{Aut}(E)\simeq \ZZ/2\ZZ$. Therefore $\Fi=\QQ(i)$ as well and $F=\QQ$. Then for each prime $p$ we can construct a form $H_p$ of $\SL_3$ over the field $\QQ_p$ as in section \ref{Section-Image of Galois} and the image of $\rho_p$ is contained and open in $H_p(\QQ_p)\cdot \QQ_p^\times$. So we get an algebraic group $H_p\cdot \mathbb{G}_m \subseteq \mathrm{Res}^{\QQ(i)_p}_{\QQ_p}\GL_3$ whose $\QQ_p$ points describe the image. 
    We know that $H_p$ is a form of $\SL_3$. Recall that it is constructed as 
    $$
    H_p = (\mathrm{Res}^{\QQ(i)_p}_{\QQ_p}\SL_3)^{tw(\mathrm{Gal}(\QQ(i)/\QQ))}
    $$
    Similar to corollary \ref{split} if $p$ is a prime that splits in $\QQ(i)$, a prime that is congruent to $1$ modulo $4$, then $H_p$ is in fact isomorphic to $\SL_3$ over $\QQ(i)_p\simeq\QQ_p\times\QQ_p$. Otherwise, it is not an inner-form and since it splits over $\QQ(i)$ it is isomorphic to the the special unitary group $\mathrm{SU}_3$ for the degree two field extension $\QQ(i)_p/\QQ_p$. So for half of the primes (primes of the form $p=4k+1$) the (Zariski closure of the) image is $\GL_3$ and for the other half (primes of the form $p=4k+3$) it is $\mathrm{SU}_3 \cdot \mathbb{G}_m$. 
\end{exmp}
\begin{exmp}
    In \cite{upton2009galois}, Upton constructs a 3-dimensional $\QQ(\zeta_3)$-rational compatible family of (motivic) Galois representations of $G_{\QQ(\zeta_3)}$ which is not self-twist or essentially self-dual and gives a precise description of its image. It is clear from her construction that these Galois representations have an outer-twist. 
    She also observes the similar phenomenon as in the last example. Namely, that for half of the primes the image is $\GL_3$ and for the other half is a unitary group. Although, we believe there is a slight error in her conclusion and the image in the latter case should be $\mathrm{SU}_3 \cdot \mathbb{G}_m$ as above, rather than the general unitary group $\mathrm{GU}_3$ as she claims. In fact since in the split case the image is $9$-dimensional (as it is $\mathrm{GL}_3$) if one believes in the Mumford-Tate conjecture, the image cannot be the $10$-dimensional group $\mathrm{GU}_3$ in the non-split case. 
    
    Even though in her case $K$ is not totally real, we can still directly apply corollary \ref{bigImCor} to a twist of the Galois representations she constructs (after a finite extension) and then deduce openness. It is easy to check the validity of this twisted Galois representation. 
\end{exmp}

\subsection{The $\GL_n$-case}
In this section we discuss the $\GL_n$-case. Everything we say here is conjectural. We assume the irreducibility conjecture (Galois representations associated with cuspidal automorphic representations are irreducible), Langlands Functoriality and the expected $p$-adic Hodge theoretic properties (de Rham with distinct Hodge-Tate weights) of our Galois representations. 
We want to see, assuming all these, when we can apply proposition \ref{validBigImage} to a regular cuspidal algebraic automorphic representation $\pi$ of $\GL_n(\AA_K)$. 
First of all we need to assume $\pi$ is neither self-twist nor essentially self-dual (in $n>2$ case). Then we only need to check that $\rho_{\pi,\lambda}$ is strongly irreducible for each $\lambda$ and that the $\overline{\QQ}_p$-Lie algebra of the image of $\rho'_{\pi,\lambda}$ is $\mathfrak{sl}_n$. 

Assume that $\rho_{\pi,\lambda}$ is reducible after restricting to $G_L$ for a finite Galois extension $L$ of $K$. The irreducible direct summands of $\rho_{\pi,\lambda}|_{G_L}$ are distinct since the Hodge-Tate weights are distinct and the action of $G_K$ on them is transitive since $\rho_{\pi,\lambda}$ is irreducible. This easily implies that $\rho_{\pi,\lambda}$ is an induction of a representation of a proper subgroup. This means that the automorphic representation $\pi$ is an induction, assuming that the automorphic induction is true. So in order to make sure that $\rho_{\pi,\lambda}$ is strongly irreducible, we only need to assume that it is not an induction.

Now let $\mathfrak{g}$ be the $\overline{\QQ}_p$-Lie algebra of the image of $\rho'_{\pi,\lambda}$. Since $\rho_{\pi,\lambda}$ and hence $\rho'_{\pi,\lambda}$ are strongly irreducible, $\mathfrak{g}$ is an irreducible Lie sub-algebra of $\mathfrak{gl}_n$ and hence reductive.
We only need to show that $\mathfrak{g}^{\mathrm{der}}=\mathfrak{sl}_n$. It is well-known that since $\mathfrak{g}^{\mathrm{der}}$ is semi-simple, there exists a semi-simple (connected) algebraic subgroup $G'$ of $\GL_n$ (over $\overline{\QQ}_p$) such that $\mathrm{Lie}(G')=\mathfrak{g}^{\mathrm{der}}$. This implies that after a finite Galois extension $M/K$, the image of $G_M$ under $\rho_{\pi,p}$ lies in $G(\overline{\QQ_p}) \subsetneq \GL_n(\overline{\QQ}_p)$ for the (connected) reductive group $G^0=G'\cdot \mathbb{G}_m\subsetneq \GL_n$. This is exactly the connected component of the $\overline{\QQ}_p$-Zariski closure of the image of $\rho_{\pi,\lambda}$. Let the whole image be $G$. Then the stabilizer of $G^0$ would give a finite Galois extension $L/K$ and the component group is isomorphic to $\mathrm{Gal}(L/K)$. The nicest situation would be if $G(\overline{\QQ_p})=G^0(\overline{\QQ_p})\rtimes \mathrm{Gal}(L/K)$. But it is not clear if this should happen. Nevertheless, one has the following result of Brion \cite{brion2015extensions}:
\begin{lemma}[Brion]
    Let $G$ be an algebraic group over a field $k$ and let
    $$
    1 \rightarrow N \rightarrow G \rightarrow Q \rightarrow 1
    $$
    be a short exact sequence (of algebraic groups over $k$) such that $Q$ is finite. Then there exists a finite subgroup $F$ of $G$, such that $G=N\cdot F$. In other words, F surjects to $Q$ and $G$ is a quotient of $N \rtimes F$ where $F$ acts on $N$ by conjugation. 
\end{lemma}

Now using this result one can at least find a finite Galois extension $M/K$ such that $G$ is the quotient of $G^0(\overline{\QQ_p})\rtimes \mathrm{Gal}(M/K)$. Now we can form the $L$-group $G^0(\overline{\QQ_p})\rtimes G_K$ where the $G_K$ action factors through the $\mathrm{Gal}(M/K)$ action from above. 
Fix a maximal torus and a Borel subgroup of $G^0$ containing it. The above action of $\mathrm{Gal}(M/K)$ on $G^0(\overline{\QQ_p})$ gives an action on the based root datum which in turn gives an action on the dual root datum. This finally gives a reductive group $H$ over $K$ which splits over $M$ whose Langlands $L$-group is $G^0(\overline{\QQ_p})\rtimes G_K$ with the above action. Now the Langlands functoriality for the $L$-map ${}^LH \rightarrow {}^L \GL_n$ implies that $\pi$ should come from an automorphic representation of the non-split reductive group $H$ via the Langlands transfer induced by the above $L$-map. This motivates the following definition:

\begin{defi}
    A regular, cuspidal, automorphic representation $\pi$ of $\mathrm{GL}_n(\AA_K)$ is said to be of \textbf{general type}, if it is neither self-twist, nor essentially self-dual (in the $n>2$ case) and there does not exists any reductive group $H$ over $K$ that is a form of a proper subgroup of $\GL_n$ such that $\pi$ is the image of an automorphic representation of $H(\AA_K)$ under the Langlands transfer attached to the $L$-map
    $$
    {}^L H \rightarrow {}^L\GL_{n,K}
    $$
\end{defi}

For instance, in the $\GL_2$ case this just means that $\pi$ is not CM and in the $\GL_3$ case it agrees with the definition \ref{3general}. Note that the condition above also automatically includes that $\pi$ is not an induction since it would be in the image of the following $L$-map then:
$$
    {}^L \mathrm{Res}^{L}_K \GL_d \rightarrow {}^L\GL_{d[L:K]}
$$

The discussion above show that if one believes in Langlands functoriality, the irreducibility conjecture and that $\rho_{\pi,\lambda}$ is de Rham with distinct Hodge-Tate weights then $\rho'_{\pi,p}$ is valid. In conclusion we make the following conjecture: 

\begin{conj}
    Let $K$ be totally real and $\pi$ be a regular, cuspidal, algebraic automorphic representation of $\GL_n(\AA_K)$ of general type. Let $E=\QQ(\pi)$ and let $F$ be the field fixed by the $E$-extra-twists of $\pi$. Then there exists a semi-simple group $H$ over $F$ which is a form of $\SL_n$ and there exists a finite extension $L/K$ such that for any prime $p$ the image of $\rho_{\pi,p}(G_L)$ is contained and open in $H(F_p)\cdot \QQ_p^\times$. 
\end{conj}

\begin{rem}
    If $\pi$ is not of general type then it comes from a smaller group $H$. Since the dimension of the group is getting smaller, there should be an optimal choice for $H$. Loosely speaking, $\pi$ should be of general type for some group. Then one has to study the image inside this smaller group, via the extra-twists for the Langlands dual of this group. Then it might be possible to give a precise description of the image as above, using the extra-twists.
\end{rem}

\section{Relations to the Mumford-Tate Group}\label{Section-Mumford-Tate}
In this section we study the relation of our results from the last section to the Mumford-Tate group and the Mumford-Tate conjecture. Almost everything in this section is conjectural, but it could give an idea of why one should believe in the conjectures presented here. 

Clozel predicts that there should be a correspondence between algebraic automorphic representations of $\GL_n(\AA_K)$ and motives over $K$ with coefficients in number fields. Let $K$ be totally real as before and $\pi$ be a regular cuspidal algebraic automorphic representation of $\GL_n(\AA_K)$. Then Clozel predicts the existence of a motive $M=M_\pi$ over $K$ with coefficients in a number field $E$ containing $\QQ(\pi)$ that is associated to $\pi$ in the way explained in \cite{clozel1991representations}. 
Here, we are thinking about a motive as a collection of realizations whose different structures are compatible through a set of comparison isomorphism. 
From now on, we assume that such a motive exists. 
Let $H_B(M)$, $H_{\mathrm{dR}}(M)$ and $H_p(M)$ be the Betti, de Rham and $p$-adic realizations of $M$. Note that the first two are $E$-vector spaces and the last one is an $E_p$-module.

Let $V=H_B(M)$.
The real and complex Betti cohomology $V \otimes_\QQ \RR$ and $V \otimes_\QQ \CC$ have an $E\otimes_\QQ \RR$ and $E\otimes_\QQ \CC = \prod_{\lambda:E\hookrightarrow \CC}\CC$ structure, respectively. Similarly the complex de Rham cohomology  $H_{\mathrm{dR}}(M)\otimes_{\QQ}\CC$ has an $E\otimes_\QQ \CC = \prod_{\lambda:E\hookrightarrow \CC}\CC$ structure.
The ($E\otimes_{\QQ}\CC$-modules) comparison isomorphism between Betti and de Rham cohomologies, 
$V_\CC \simeq H_{\mathrm{dR}}(M)\otimes_{\QQ}\CC$, equips $V$ with a rational Hodge structure. We denote this Hodge structure with
$$
h_{\pi}:\mathbb{S}\rightarrow \GL(V_\RR)
$$
where $\mathbb{S}=\mathrm{Res}^{\CC}_\RR\mathbb{G}_m$ is the Deligne torus. Fixing an $E$-basis for $V_\QQ$ which in turn gives an $E\otimes \RR$ basis for $V_\RR$ enables us to write this as 
$$
h_{\pi,\infty}:\mathbb{S}\rightarrow \GL_n(E\otimes_{\QQ}\RR)
$$
This representation should be thought of as the analogue of our $p$-adic Galois representations $\rho_{\pi,p}:G_K\rightarrow \GL_n(E\otimes_{\QQ}\QQ_{p})$ from section \ref{Section-Automorphic}, associated to the prime at infinity. 
Note that this is equipped with an action of $\mathrm{Aut}(E)$ on the coefficients. 

Recall that the Mumford-Tate group of the motive $M$ is defined to be the smallest algebraic group $G \subseteq \GL(V_\QQ)$ over $\QQ$ that contains the image of $h_\pi$ (the map $h_\pi$ factors through $G_\RR$). The Mumford-Tate conjecture then states that the connected component of the $\QQ_p$-Zariski closure of the image of the Galois representation $H_p(M)$ is equal to $G\times_{\QQ}\QQ_p$. 

Now, let $\pi$ be of general type and $\Gamma$ be the group of $E$-extra-twists of $\pi$. Let $|.|^m\omega$ be the central character of $\pi$ where $\omega$ is a finite order Hecke character. 
From now on, for simplicity we assume that $m$ is divisible by $n$. So let $m=nd$. Then the outer-twists of $\pi$ are of the form $(\tau, |\cdot|^{2d}\eta)$ for a finite character $\eta$ and hence the outer-twists of $\rho_{\pi,p}$ are of the form $(\tau, \epsilon_p^{2d}\eta)$ where $\epsilon_p$ is the $p$-adic cyclotomic character and we think of $\eta$ as a finite Galois character. The extra-twist of $\pi$ then induce extra-twists on the motive $M_\pi$. An inner-twist $(\sigma,\chi)$ induces an isomorphism ${}^\sigma M_\pi \simeq M_\pi \otimes \chi$ where $\chi$ is the Artin motive associated with the finite character $\chi$. The outer-twist 
$(\tau, |\cdot|^{2d}\eta)$ induces an isomorphism ${}^\tau M_\pi \simeq M_\pi^{\vee} \otimes \QQ(2d) \otimes \eta$. 

In particular, the extra-twists also induce symmetries on the Hodge-structure since $E$ acts on the motive $M$ via endomorphisms. Twisting with finite characters does not affect the Hodge structure and twisting with the $2d$'th power of the cyclotomic character amounts to twisting with the Tate Hodge structure $\QQ(2d)$. This means that for each inner-twist $\sigma \in \Gi$ one has ${}^\sigma h_\pi \simeq h_\pi$ and for each outer-twist $\tau \in \Go$ one has ${}^\tau h_\pi \simeq h_{\pi}^{\vee} \otimes_{\QQ} \QQ(2d)$. 
Now, if we twist $h_{\pi}$ with $\QQ(-d)$ we still get ${}^\sigma h_\pi(-d) \simeq h_\pi(-d)$ for each inner-twist and ${}^\tau h_\pi(-d) \simeq h_{\pi}(-d)^{\vee}$ for each outer-twist. This is analogous to the representation $\rho'_{\pi,p}$ from section \ref{Section-Automorphic}.

Since an isomorphism of rational Hodge structures comes from an isomorphism over $\QQ$ between the underlying rational vector spaces, and since everything is compatible with the $E$-structures, we can find matrices $\alpha_\sigma, \alpha_\tau \in \GL_n(E)$ that give the isomorphisms above by conjugation. so we get: 
$$
 \left\{\begin{matrix}
h_{\pi,\infty}(-d) = \alpha_{\sigma} {}^\sigma h_{\pi,\infty}(-d) \alpha_{\sigma}^{-1}
\\
h_{\pi,\infty}(-d) = \alpha_{\tau} {}^\tau h_{\pi,\infty}^{-T}(-d) \alpha_{\tau}^{-1}
\end{matrix}\right.
$$
Note that the determinant of $h_{\pi,\infty}(-d)$ is trivial so it has values in $\SL_n$.
Now define the $1$-cocycle $f:\Gamma \rightarrow \mathrm{Aut}_E(\SL_n)$ by sending an inner-twist  $\sigma$ to $\mathrm{ad}(\alpha_\sigma)$ and an outer-twist $\tau$ to $\mathrm{ad}(\alpha_\tau)\circ (\cdot)^{-T}$. Then as in section \ref{form}, we can define the twisted action of $\Gamma$ on $\mathrm{Res}^E_F\SL_n$ and the matrices in the image of $h_{\pi,\infty}(-d)$ are clearly invariant under this action. We define the groups $H_\infty:=(\mathrm{Res}^{E\otimes \RR}_{F\otimes \RR}\SL_n)^{tw(\Gamma)}$ and $H:=(\mathrm{Res}^E_F\SL_n)^{tw(\Gamma)}$. Note that $H_\infty$ is the base change of $H$ to $\RR$ and it is the Archimedean analogue of the groups $H_p$ from section \ref{Section-Automorphic}.
\begin{rem}
    We could also define the group $H_\infty$ without assuming the existence of the motive, only from the real Hodge structure coming from the Archimedean part of $\pi$. But then the connection to the Mumford-Tate group is of course less clear.
\end{rem}

\begin{lemma}\label{MTcontained}
    The Mumford-Tate group of the motive $M_\pi$ is contained in the group $\mathrm{Res}^F_\QQ(H)\cdot \mathbb{G}_m$. 
\end{lemma}
\begin{proof}
    The image of $h_{\pi,\infty}(-d)$ lies in $H_\infty$ by the definition of $H_\infty$. This implies that the map $h_{\pi,\infty}$ factors through $H_\infty \cdot \mathbb{G}_{m,\RR}$ and therefore the map $h_{\pi}$ factors through the base change of the group $\mathrm{Res}^F_\QQ(H)\cdot \mathbb{G}_{m,\QQ}$ to $\RR$. This implies the result.
\end{proof}
Note that the dimension of the group $H$ is equal to the dimension of all groups $H_p$ from section \ref{Section-Automorphic}. Our results in that section make it reasonable to make the following conjecture. 
\begin{conj}\label{MTconj}
    If $\pi$ is of general type  then the Mumford-Tate group of $M_\pi$ is equal to $\mathrm{Res}^F_\QQ(H)\cdot \mathbb{G}_{m,\QQ}$. 
\end{conj}

We showed that the group $H$ (or any of the $H_p$'s from section \ref{Section-Automorphic}) gives an upper bound for the Mumford-Tate group. On the other hand,
in the special case that $M$ is an abelian motive, Deligne shows that the Mumford-Tate group $\mathrm{MT(M_\pi)}$, after base changing to $\QQ_p$, always contains the connected component of the image of the $p$-adic Galois representations. So if we known that the image is open in $H_p$ even for one prime, the conjecture above follows. In particular we have the following result in the $n=2$ case (which was probably known to Nekovar): 
\begin{cor}
    Let $f$ be a non-CM Hilbert modular newform of parallel weight 2 over the totally real field $K$ and assume there is an abelian variety $A_f$ associated to it. Then  conjecture \ref{MTconj} and the Mumford-Tate conjecture hold for $A_f$. 
\end{cor}\begin{proof}
    Note that $\mathrm{MT}(M_\pi)\otimes_{\QQ}\QQ_p$ is contained in $(\mathrm{Res}^F_\QQ(H)\cdot \mathbb{G}_m) \otimes \QQ_p$ by lemma \ref{MTcontained} and contains the connected component of the Zariski-closure of the image of $\rho_{f,p}$ which is equal to $\mathrm{Res}^{F_p}_{\QQ_p}H_p(\QQ_p)\cdot\QQ_p^{\times}$ by \ref{bigGL2}. Since the dimensions match, both must be equality.
\end{proof}

At the end we want to come back to remark \ref{remRou}. Recall that by \cite{chenevier2014p}, for each prime $p$ there exist an Azumaya algebra $D_p$ over $\QQ(\pi)_p:=\QQ(\pi)\otimes_{\QQ}\QQ_p$ such that the Galois representation $\rho_{\pi,p}$ factors through 
$D_{p} \subset D_{p} \otimes_{\QQ(\pi)} E \simeq M_n(E_p)$, assuming irreducibility of the Galois representation. We are interested to see if the local object $D_p$ should be a global object $D$ defined over $\QQ(\pi)$. 
In the $n=1$ case this is clear. In the $n=2$ case, since $\rho_{\pi,p}$ is odd, the residual representation is multiplicity free and a result of Bellaïche and Chenevier \cite{bellaiche2009families} shows that in this case every pseudo-representation can be defined over its trace field and hence $D_p \simeq \GL_2$ for all primes $p$. So $D\simeq \GL_2$ works. 

Now assume that $\pi$ is of general type and assume conjecture \ref{MTconj} and the Mumford-Tate conjecture. First notice that for every $\sigma \in \mathrm{Gal}(E/\QQ(\pi))$ we have an inner-twist $\pi \simeq {}^\sigma\pi$ therefore $F \subseteq \QQ(\pi)$. 
Now we know by our assumptions that for some finite Galois extension $L/K$ the image of $\rho_{\pi,p}|_{G_L}$ is open in $H_p(F_p)\cdot \QQ_p^{\times}$. This is the $\QQ_p$-Zariski closure of the image hence the $F_p$-Zariski closure of the image is $H_p(F_p)\cdot F_p^{\times}$. Note that the inner product is happening in $\GL_n$ so since $H_p$ is a form of $\SL_n$ we deduce that $G_p:=H_p \cdot \mathbb{G}_{m,F_p}$ is a form of $\GL_n$. Since the image is contained in $D_p^{\times}$ which is an algebraic group over $\QQ(\pi)_p$ and $F_p \subseteq \QQ(\pi)_p$, we have 
$$
G_p (F_p) \subseteq D_p^{\times}(\QQ(\pi)_p)\subseteq \GL_n(E_p)
$$
and the base change of either $G_p$ or $D_p^\times$ to $E_p$ is equal to $\GL_n$. This implies that $D_p^{\times}$ is the $\QQ(\pi)_p$-Zariski closure of the image of the Galois representation. In particular, by the Mumford-Tate conjecture, the groups $D_p^\times$ should come from a Global group, namely the $\QQ(\pi)$-Zariski closure of the image of the map $h_{\pi,\infty}$. 
This is an inner-form of $\GL_n$ since all $G_p$'s become inner forms over $F_p^\mathrm{inn}$ so it is the group of units of a central simple algebra $D$ over $\QQ(\pi)$. 

In particular in the $n=3$ case, $\pi$ is either an induction of a character from a degree 3 extension, essentially $\mathrm{sym}^2$ or of general type. In the first two cases $D_p$ being global reduces to the $n=1$
 and $n=2$ case and in the third case follows from the discussion above (assuming all the above conjectures). This makes it reasonable to make the following conjecture:
\begin{conj}
    Let $\pi$ be a regular cuspidal algebraic automorphic representation of $\GL_n(\AA_K)$ and $\rho_{\pi,p}$ the associated Galois representation and $D_p$ the Azumaya algebra coming from \cite{chenevier2014p}. Then there exists a central simple algebra $D$ over $\QQ(\pi)$ such that $D_p \simeq D \otimes_{\QQ(\pi)}\QQ(\pi)_p$ for every $p$.
\end{conj}

In particular this conjecture implies that for all but finitely many primes, the representation $\rho_{\pi,p}$ can be defined over its trace field.

\bibliographystyle{amsplain}
\bibliography{refs}

\end{document}